\documentclass[11pt]{article}

\usepackage{amsmath, amsthm, amssymb}
\usepackage{enumerate}
\usepackage{pdflscape}
\usepackage{caption}
\usepackage{bm}

\usepackage{ifpdf}
\ifpdf
\usepackage[pdftex]{graphicx}
\else
\usepackage[dvips]{graphicx}
\fi
\usepackage{tikz}
 	 \usetikzlibrary{arrows,backgrounds}
\usepackage[all]{xy}

\usepackage{multicol}

\usepackage{tocvsec2}

\usepackage{bbm}

\input xy
\xyoption{all}

\usepackage[pdftex,plainpages=false,hypertexnames=false,pdfpagelabels]{hyperref}
\newcommand{\arxiv}[1]{\href{http://arxiv.org/abs/#1}{\tt arXiv:\nolinkurl{#1}}}
\newcommand{\arXiv}[1]{\href{http://arxiv.org/abs/#1}{\tt arXiv:\nolinkurl{#1}}}
\newcommand{\doi}[1]{\href{http://dx.doi.org/#1}{{\tt DOI:#1}}}

\newcommand{\googlebooks}[1]{(preview at \href{http://books.google.com/books?id=#1}{google books})}

\usepackage{xcolor}
\definecolor{dark-red}{rgb}{0.7,0.25,0.25}
\definecolor{dark-blue}{rgb}{0.15,0.15,0.55}
\definecolor{medium-blue}{rgb}{0,0,.8}
\definecolor{DarkGreen}{RGB}{0,150,0}
\definecolor{rho}{named}{red}
\hypersetup{
   colorlinks, linkcolor={purple},
   citecolor={medium-blue}, urlcolor={medium-blue}
}

\usepackage{longtable}
\usepackage{fullpage}

\theoremstyle{plain}
\newtheorem{thm}{Theorem}[section]
\newtheorem*{thm*}{Theorem}
\newtheorem{thmalpha}{Theorem}

\newtheorem{cor}[thm]{Corollary}
\newtheorem{coralpha}[thmalpha]{Corollary}
\newtheorem*{cor*}{Corollary}

\newtheorem*{conj*}{Conjecture}
\newtheorem{lem}[thm]{Lemma}
\newtheorem{fact}[thm]{Fact}

\newtheorem{prop}[thm]{Proposition}

\newtheorem*{quest*}{Question}
\newtheorem*{claim*}{Claim}

\theoremstyle{definition}
\newtheorem{defn}[thm]{Definition}

\newtheorem{alg}[thm]{Algorithm}

\newtheorem{nota}[thm]{Notation}

\newtheorem{ex}[thm]{Example}
\newtheorem{example}[thm]{Example}

\newtheorem{sub-ex}[thm]{Sub-Example}
\newtheorem{rem}[thm]{Remark}
\newtheorem*{rem*}{Remark}
\newtheorem{remark}[thm]{Remark}

\newtheorem{warn}[thm]{Warning}

\DeclareMathOperator{\Aut}{Aut}
\DeclareMathOperator{\coev}{coev}
\DeclareMathOperator{\Dim}{Dim}

\DeclareMathOperator{\End}{End}
\DeclareMathOperator{\ev}{ev}
\DeclareMathOperator{\gr}{gr}
\DeclareMathOperator{\Hom}{Hom}
\DeclareMathOperator{\Mat}{Mat}

\DeclareMathOperator{\id}{id}

\DeclareMathOperator{\im}{im}
\DeclareMathOperator{\Irr}{Irr}

\DeclareMathOperator{\Tr}{Tr}
\DeclareMathOperator{\tr}{tr}

\newcommand{\comment}[1]{}

\newcommand{\be}{\begin{enumerate}[label=(\arabic*)]}
\newcommand{\ee}{\end{enumerate}}

\def\semicolon{;}
\def\applytolist#1{
    \expandafter\def\csname multi#1\endcsname##1{
        \def\multiack{##1}\ifx\multiack\semicolon
            \def\next{\relax}
        \else
            \csname #1\endcsname{##1}
            \def\next{\csname multi#1\endcsname}
        \fi
        \next}
    \csname multi#1\endcsname}

\def\calc#1{\expandafter\def\csname c#1\endcsname{{\mathcal #1}}}
\applytolist{calc}QWERTYUIOPLKJHGFDSAZXCVBNM;
\def\bbc#1{\expandafter\def\csname bb#1\endcsname{{\mathbb #1}}}
\applytolist{bbc}QWERTYUIOPLKJHGFDSAZXCVBNM;
\def\bfc#1{\expandafter\def\csname bf#1\endcsname{{\mathbf #1}}}
\applytolist{bfc}QWERTYUIOPLKJHGFDSAZXCVBNM;
\def\sfc#1{\expandafter\def\csname s#1\endcsname{{\sf #1}}}
\applytolist{sfc}QWERTYUIOPLKJHGFDSAZXCVBNM;
\def\fc#1{\expandafter\def\csname f#1\endcsname{{\mathfrak #1}}}
\applytolist{fc}QWERTYUIOPLKJHGFDSAZXCVBNM;

\newcommand{\Fun}{{\sf Fun}}

\newcommand{\Bim}{{\sf Bim}}
\newcommand{\bfBim}{{\sf Bim_{bf}}}

\newcommand{\fdVec}{{\sf Vec_{fd}}}
\newcommand{\Hilb}{{\sf Hilb}}
\newcommand{\fdHilb}{{\sf Hilb_{fd}}}

\newcommand{\noshow}[1]{}
\newcommand{\MR}[1]{}

\newcommand{\Cstar}{{\rm C^*}}

\usetikzlibrary{shapes}
\usetikzlibrary{cd}
\usetikzlibrary{backgrounds}
\usetikzlibrary{decorations,decorations.pathreplacing,decorations.markings}
\usetikzlibrary{fit,calc,through}
\usetikzlibrary{external}
\usetikzlibrary{arrows}
\tikzset{vertex/.style = {shape=circle,draw,fill=black,inner sep=0pt,minimum size=5pt}}
\tikzset{edge/.style = {->,> = latex', bend right}}
\tikzset{
	super thick/.style={line width=3pt}
}
\tikzset{
    quadruple/.style args={[#1] in [#2] in [#3] in [#4]}{
        #1,preaction={preaction={preaction={draw,#4},draw,#3}, draw,#2}
    }
}
\tikzstyle{shaded}=[fill=red!10!blue!20!gray!30!white]
\tikzstyle{unshaded}=[fill=white]
\tikzstyle{empty box}=[circle, draw, thick, fill=white, opaque, inner sep=2mm]
\tikzstyle{annular}=[scale=.7, inner sep=1mm, baseline]
\tikzstyle{rectangular}=[scale=.75, inner sep=1mm, baseline=-.1cm]
\tikzstyle{mid>}=[decoration={markings, mark=at position 0.5 with {\arrow{>}}}, postaction={decorate}]
\tikzstyle{mid<}=[decoration={markings, mark=at position 0.5 with {\arrow{<}}}, postaction={decorate}]
\tikzstyle{over}=[double, draw=white, super thick, double=]
\tikzstyle{box} = [rectangle,draw,rounded corners=5pt,very thick]

\newcommand{\roundNbox}[6]{
	\draw[rounded corners=5pt, very thick, #1] ($#2+(-#3,-#3)+(-#4,0)$) rectangle ($#2+(#3,#3)+(#5,0)$);
	\coordinate (ZZa) at ($#2+(-#4,0)$);
	\coordinate (ZZb) at ($#2+(#5,0)$);
	\node at ($1/2*(ZZa)+1/2*(ZZb)$) {#6};
}

\begin{document}
\title{Unitary dual functors for unitary multitensor categories}
\author{David Penneys}
\date{\today}
\maketitle
\begin{abstract}
We classify which dual functors on a unitary multitensor category are compatible with the dagger structure in terms of groupoid homomorphisms from the universal grading groupoid to $\mathbb{R}_{>0}$ where the latter is considered as a groupoid with one object.
We then prove that all unitary dual functors induce unitarily equivalent bi-involutive structures.
As an application, we provide the unitary version of the folklore correspondence between shaded planar ${\rm C^*}$ algebras with finite dimensional box spaces and unitary multitensor categories with a chosen unitary dual functor and chosen generator.
We make connection with the recent work of Giorgetti-Longo to determine when the loop parameters in these planar algebras are scalars.
Finally, we show that we can correct for many non-spherical choices of dual functor by adding the data of a spherical state on $\End_{\mathcal{C}}(1_{\mathcal{C}})$, similar to the spherical state for a graph planar algebra.
\end{abstract}

\section{Introduction}

In a rigid monoidal category $\cC$, 
every object has a \emph{dual}, consisting of a triple $(c^\vee, \ev_c, \coev_c)$ where $c^\vee \in \cC$ and $\ev_c \in\cC(c^\vee\otimes c \to 1_\cC)$ and $\coev_c\in \cC(1_\cC \to c\otimes c^\vee)$ satisfy the \emph{zig-zag axioms}:
$$
\begin{tikzpicture}[baseline=-.1cm, xscale=-1]
	\draw (-.6,-.6) -- (-.6,0) arc (180:0:.3cm) arc (-180:0:.3cm) -- (.6,.6);
	\node at (-.8,-.2) {\scriptsize{$c$}};
	\node at (-.25,0) {\scriptsize{$c^\vee$}};
	\node at (.8,.2) {\scriptsize{$c$}};
\end{tikzpicture}
:=
(\id_c \otimes \ev_c)\circ(\coev_c\otimes \id_c)
=
\id_c
=:
\begin{tikzpicture}[baseline=-.1cm]
	\draw (0,-.6) -- (0,.6);
	\node at (.2,0) {\scriptsize{$c$}};
\end{tikzpicture}
\qquad\qquad
\begin{tikzpicture}[baseline=-.1cm]
	\draw (-.6,-.6) -- (-.6,0) arc (180:0:.3cm) arc (-180:0:.3cm) -- (.6,.6);
	\node at (-.8,-.2) {\scriptsize{$c^\vee$}};
	\node at (.15,0) {\scriptsize{$c$}};
	\node at (.8,.2) {\scriptsize{$c^\vee$}};
\end{tikzpicture}
=
\begin{tikzpicture}[baseline=-.1cm]
	\draw (0,-.6) -- (0,.6);
	\node at (.2,0) {\scriptsize{$c^\vee$}};
\end{tikzpicture}
=
\id_{c^\vee},
$$
and every object is isomorphic to the dual of some other object.
By choosing a dual for each $c\in \cC$, we get an anti-monoidal  \emph{dual functor} $\vee : \cC \to \cC$ defined on a morphism $f\in \cC(a\to b)$ by
$$
f^\vee 
:=
\begin{tikzpicture}[baseline=-.1cm]
	\draw (0,.3) arc (0:180:.3cm) -- (-.6,-.8);
	\draw (0,-.3) arc (-180:0:.3cm) -- (.6,.8);
	\roundNbox{fill=white}{(0,0)}{.3}{0}{0}{$f$}
	\node at (-.8,-.6) {\scriptsize{$d^\vee$}};
	\node at (.8,.6) {\scriptsize{$c^\vee$}};
\end{tikzpicture}
=
(\ev_d\otimes \id_{c^\vee})
\circ
(\id_{d^\vee}\otimes f\otimes \id_{c^\vee})
\circ
(\id_{d^\vee}\otimes \coev_c).
$$
A dual functor comes with
a canonical anti-monoidal tensorator
$$
\nu_{a,b} 
:=
\begin{tikzpicture}[baseline=-.1cm]
	\draw (.8,.6) -- (.8,0) arc (0:-180:.4cm) arc (0:180:.2cm) -- (-.4,-.6);
	\draw (.7,.6) -- (.7,0) arc (0:-180:.3cm) arc (0:180:.5cm) -- (-.9,-.6);
	\node at (1.4,.4) {\scriptsize{$(a\otimes b)^\vee$}};
	\node at (-.6,-.4) {\scriptsize{$a^\vee$}};
	\node at (-1.1,-.4) {\scriptsize{$b^\vee$}};
\end{tikzpicture}
=
(\ev_b\otimes \id_{(a\otimes b)^\vee})
\circ 
(\id_{b^\vee}\otimes \ev_a \otimes \id_b \otimes \id_{(a\otimes b)^\vee})
\circ
(\id_{b^\vee\otimes a^\vee} \otimes \coev_{a\otimes b}),
$$
and any two dual functors $\vee_1,\vee_2$ are \emph{uniquely} monoidally naturally isomorphic via 
$$
\zeta_c
:=
\begin{tikzpicture}[baseline=-.1cm]
	\draw[thick, red] (-.4,-.4) -- (-.4,0) arc (180:0:.2cm);
	\draw[thick] (.4,.4) -- (.4,0) arc (0:-180:.2cm);
	\node at (-.7,-.2) {\scriptsize{$c^{\vee_2}$}};
	\node at (.7,.2) {\scriptsize{$c^{\vee_1}$}};
\end{tikzpicture}
= 
(\ev_c^2 \otimes \id_{c^{\vee_1}})\circ (\id_{c^{\vee_2}} \otimes \coev_c^1).
$$
A \emph{pivotal structure} on $\cC$ is a pair $(\vee, \varphi)$, where $\vee$ is a chosen dual functor, and $\varphi: \id \Rightarrow \vee\circ\vee$ is a monoidal natural isomorphism.
Using $\varphi$, one can define the left and right \emph{quantum dimension} of an object $c\in \cC$; we refer the reader to \S\ref{sec:PivotalStructures} for a detailed discussion of pivotal structures.

The above definitions are poorly behaved in the context of rigid tensor $\Cstar$ categories.
More precisely, the above definitions fail the \emph{principle of equivalence} \cite{nlab:principle_of_equivalence}, which roughly states that mathematical definitions should be invariant under the proper notion of equivalence.
As a basic example, when one works with Hilbert spaces, the correct notion of equivalence is that of \emph{unitary isomorphism}, i.e., bounded linear maps $u: H\to K$ such that $u^*u = \id_H$ and $uu^* = \id_K$, and \emph{not} bouned linear isomorphism.\footnote{
Conjugating a self-adjoint operator by a bounded linear isomorphism need not produce a self-adjoint operator unless the isomorphism is unitary.
Similarly, in finite dimensions, taking coordinates for a self-adjoint opertator with respect to a basis need not produce a self-adjoint matrix unless the basis is orthonormal.
In this respect, the notions of linear isomorphism and basis fail the principle of equivalence for the $\Cstar$ category $\fdHilb$ of finite dimensional Hilbert spaces, while the notions of unitary isomorphism and orthonormal basis satisfy the principle of equivalence.
}
As $\Cstar$ categories admit an equivalent definition as those categories which admit a faithful dagger functor to the category $\Hilb$ of Hilbert spaces which is norm-closed on the level of hom spaces \cite{MR808930}, we see that one must work with dagger functors and unitary (natural) isomorphisms to satisfy the principle of equivalence for dagger categories.

Indeed, a dual functor on a rigid tensor $\Cstar$ category need not be a dagger functor, and the canonical tensorator $\nu$ need not be unitary.
With this in mind, we call a dual functor $\vee: \cC \to \cC$ \emph{unitary} if it is a dagger tensor functor, i.e., for all $a,b\in \cC$ and $f\in \cC(a\to b)$, the canonical tensorator $\nu_{a,b}$ is unitary and $f^{\vee\dag} = f^{\dag \vee}$. 
Given a unitary dual functor $\vee: \cC \to \cC^{\text{mop}}$, there is a unique pivotal structure for which left and right dimensions of objects are positive\,\footnote{
\label{footnote:Pseudounitary}
We call a pivotal structure \emph{pseudounitary} if all quantum dimensions of objects are strictly positive.
This definition is equivalent to \cite[Def.~9.4.4]{MR3242743} for fusion categories by uniqueness of the Frobenius-Perron dimensions.
} 
and given by $\ev^\dag_c \circ \ev_c$ and $\coev_c \circ \coev^\dag_c$ respectively:
\begin{equation}
\label{eq:CanonicalPivotalStructureFromUnitaryDualFunctor}
\varphi_c 
:=
(\coev_{c}^\dag \otimes \id_{c^{\vee\vee}})\circ (\id_c\otimes \coev_{c^\vee})
=
(\id_{c^{\vee\vee}} \otimes \ev_c)\circ ( \ev_{c^\vee}^\dag \otimes \id_c).
\end{equation}
By \cite[Lem.~7.5]{MR2767048} (which is Proposition \ref{prop:VeeAndDagCompatibility} below), a dual functor $\vee$ is unitary if and only if $\varphi$ defined as in \eqref{eq:CanonicalPivotalStructureFromUnitaryDualFunctor} above defines a pivotal structure.
We call such pivotal structures \emph{unitary}, but one should really only consider the term `unitary pivotal structure' as a synonym for `the canonical pivotal structure associated to a unitary dual functor' as in \cite[\S7.3]{MR2767048}.

Unitary dual functors on rigid tensor $\Cstar$ categories were first constructed in \cite{MR1444286,MR2091457,MR3342166}.
The notion of a quantum dimension for dualizable objects in a tensor $\Cstar$ category with simple unit object was established in \cite{MR1444286} via \emph{standard solutions to the conjugate equations}.
In \cite{MR2091457}, it was further clarified that for a \emph{unitary tensor category} $\cC$, which is an idempotent complete rigid tensor $\Cstar$ category with simple unit object, for every object $c\in \cC$, there is a unique \emph{balanced} dual $(\overline{c}, \ev_c, \coev_c)$ up to unique unitary isomorphism satisfying the zig-zag axioms and the balancing equation:
$$
\ev_c \circ (\id_{\overline{c}} \otimes f) \circ \ev_c^\dag
=
\coev_c^\dag \circ (f\otimes \id_{\overline{c}})\circ \coev_c
\qquad
\forall c\in \cC, f\in \cC(c\to c).
$$
Moreover, choosing these balanced duals gives gives a canonical unitary dual functor whose associated unitary pivotal structure is \emph{spherical}.
This result was later expanded in \cite{MR3342166}, in the context of von Neumann algebras with finite dimensional centers, to \emph{unitary multitensor categories}, which are idempotent complete rigid tensor $\Cstar$ categories.
For a unitary multitensor category $\cC$, $1_\cC$ is no longer simple; however, since $\cC$ is automatically semisimple by a generalization of \cite[Lem.~3.9]{MR1444286}, $1_\cC$ decomposes as an orthogonal finite direct sum of simples $1_\cC = \bigoplus_{i=1}^r 1_i$.
Each `corner' $\cC_{ii}:=1_i \otimes \cC \otimes 1_i $ is again a unitary tensor category.

While the existence of this canonical unitary dual functor and spherical structure for a unitary multitensor category is extremely powerful, it is not always the most relevant unitary dual functor for applications.
A first example is the unitary tensor category $\bfBim(R)$ of bifinite bimodules over the hyperfinite ${\rm II}_1$ factor $R$.
The most widely used unitary dual functor on $\bfBim(R)$ is built from the canonical trace on $R$ via left and right $R$-valued inner products on the subspaces of bounded vectors (see \cite{MR1424954,MR3040370,ClaireSorinII_1,1704.02035}).
Often, one restricts to \emph{spherical/extremal} bimodules where the canonical unitary dual functor agrees with this tracial one.

Notice $\bfBim(R)$ admits a grading by $\bbR_{>0}$ given by the ratio of left to right von Neumann dimension.
Whether $\bbR_{>0}$ \emph{is} the universal grading group of $\bfBim(R)$ is a tantalizing open question.
However, this grading is sufficient to understand the difference between the tracial unitary dual functor and unitary pivotal structure, which corresponds to the identity group homomorphism $\bbR_{>0} \to \bbR_{>0}$, and the canonical unitary spherical structure, which corresponds to the trivial group homomorphism under our Theorem \ref{thm:Main} below.
We refer the reader to Example \ref{ex:SphericalAndTracialPivotalStructures} for more details.

A second example is the industry of constructing subfactor planar algebras as planar subalgebras of graph planar algebras \cite{MR1929335,MR2511128,MR2679382,MR2979509,MR2822034,MR3394622,MR3314808,MR3402358,MR3306607,EH3}. 
(Such a realization is always possible for finite depth subfactor planar algebras by \cite{MR2812459}, although this result is not necessary in the construction.
See also \cite{EH3,ModuleEmbedding} for the module embedding theorem.)
By Example \ref{ex:GPA} below, the projection category of the planar algebra of a finite connected bipartite graph $\Gamma$ is dagger equivalent to $\End^\dag(\fdHilb^n)$, the unitary multitensor category of dagger endofunctors of $n$ copies of finite dimensional Hilbert spaces, where $n$ is the number of vertices of $\Gamma$.
The planar algebra gives a particular unitary pivotal structure related to Frobenius-Perron data of $\Gamma$, which does \emph{not} correspond to the canonical unitary spherical structure.
However, one has a canonical \emph{spherical state} on the graph planar algebra \cite[Prop.~3.4]{MR1865703}, which implies any \emph{evaluable} planar subalgebra is a subfactor planar algebra \cite[\S8]{MR1929335}.
We refer the reader to \S\ref{sec:SphericalStates} for more details.

The relevant unitary dual functor and corresponding pivotal structure to explain this second class of examples is provided by \cite{1805.09234} in the context of 2-$\Cstar$-categories, which defines standard/minimal solutions to the conjugate equations \emph{with respect to a particular object $X\in \cC$}.
While providing deeper insight into well-behaved choices of solutions to the conjugate equations, they leave open 
the important 
question of classifying all unitary dual functors.
Notice that although two unitary dual functors are uniquely monoidally naturally isomorphic, this natural isomorphism need not be unitary!
(The unique monoidal natural isomorphism may fail the principle of equivalence for tensor $\Cstar$ categories.)
Hence two unitary dual functors need not be unitarily equivalent.

In this article, we prove the following classification theorem.

\begin{thmalpha}
\label{thm:Main}
Let $\cC$ be a unitary multitensor category.
There are canonical bijections between:
\begin{enumerate}[(1)]
\item
Pseudounitary\,\footnote{
See Footnote \ref{footnote:Pseudounitary} on the previous page for the definition of a pseudounitary pivotal structure.
}
 pivotal structures up to monoidal natural isomorphism.\footnote{
If two pseudounitary pivotal structures are monoidally naturally isomorphic, they are so in a unique way.
}
\item
Unitary dual functors up to unitary monoidal natural isomorphism.\footnote{
If two unitary dual functors are unitarily monoidally naturally isomorphic, they are so in a unique way.
}

\item
Groupoid homomorphisms $\cU \to \bbR_{>0}$, where the latter is viewed as a groupoid with one object.
\end{enumerate}
\end{thmalpha}

Here, $\cU$ is the \emph{universal grading groupoid} of $\cC$, which is defined analogously to the universal grading group of a tensor category as in \cite[\S4.14]{MR3242743} (see \S\ref{sec:GradingGroupoid} below for the definition).
We show in Lemma \ref{lem:RatioHomormophism} that from a pseudounitary pivotal structure $\varphi$, we get a groupoid homomorphism by taking ratios of dimensions of simple objects; that is, if a simple $c\in \cC$ is graded by a morphism $g\in \cU$, we get a well defined $\pi \in \Hom(\cU \to \bbR_{>0})$ by 
$$
\pi(g):=
\frac{\dim^\varphi_L(c)}{\dim^\varphi_R(c)}.
$$
Conversely, 
following a suggestion of Andr\'{e} Henriques,
given a $\pi \in \Hom(\cU \to \bbR_{>0})$, we define a canonical \emph{$\pi$-balanced dual functor} $\vee_\pi$ which is unique up to unique unitary monoidal natural isomorphism by finding \emph{$\pi$-balanced solutions} $(\ev^\pi_c, \coev^\pi_c)$ to the conjugate equations which satisfy the zig-zag axioms and 
for all $f\in \cC(c \to c)$ and morphisms $g\in \cU$,
$$
\Psi\left(
\ev^\pi_c \circ (\id_{c^{\vee_\pi}}\otimes f_{g}) \circ (\ev^\pi_c)^\dag
\right)
=
\pi(g)
\cdot
\Psi\left(
(\coev^\pi_c)^\dag \circ (f_{g}\otimes \id_{c^{\vee_\pi}}) \circ \coev^\pi_c
\right)
$$
where $\Psi$ the linear functional in $\cC(1\to 1) \to \bbC$ which sends every minimal projection to $1_\bbC$, $f_g\in \cC(c_g \to c_g)$ is the $g$-homogeneous component of $f$, and $c_g$ is the $g$-graded component of $c\in \cC$.
This proof is similar to \cite[Lem.~3.9]{MR2091457} and \cite[Prop.~2.2.15]{MR3204665}.

Unitary dual functors and pivotal structures are closely related to the more general notion of \emph{bi-involutive structure} from \cite[\S2.1]{MR3663592}.
An \emph{involution} on a multitensor category \cite{MR2861112} is a conjugate-linear anti-monoidal functor $(\overline{\,\cdot\,}, \nu) : \cC \to \cC$ together with a monoidal natural isomorphism $\varphi : \id_\cC \Rightarrow \overline{\overline{\,\cdot\,}}$. 
When $\cC$ is unitary, we call $(\overline{\,\cdot\,}, \nu,\varphi)$ \emph{bi-involutive} if $(\overline{\,\cdot\,}, \nu)$ is an anti-monoidal dagger functor and $\varphi$ is unitary.
One obtains a bi-involutive structure from a unitary dual functor $\vee$ and its canonical unitary pivotal structure $\varphi$ by simply forgetting the evaluation and coevaluation maps.

Motivated by the example $\bfBim(R)$ above and \cite[Rem.~2.14]{1704.02035} (see also Example \ref{ex:SphericalAndTracialPivotalStructures}), we prove the following somewhat surprising result in \S\ref{sec:Bi-involutive}.

\begin{coralpha}
\label{cor:UniqueBi-involutive}
Any two bi-involutive structures on a unitary multitensor category induced by unitary dual functors are canonically unitarily equivalent.
\end{coralpha}

As an application of Theorem \ref{thm:Main}, we now understand the unitary version of the folklore correspondence between shaded planar algebras and pivotal multitensor categories with a choice of generator \cite{MR2559686,MR2811311,1207.1923,MR3405915,1607.06041}.

\begin{thmalpha}
\label{thm:UnitaryPAequivalence}
There is an equivalence of categories\,\footnote{
Here we suppress a subtlety about contractible 2-categories; see Footnote \ref{footnote:Truncate2Category} in Theorem \ref{thm:PAEquivalence} for details.} 
\[
\left\{\, 
\parbox{4.5cm}{\rm Shaded planar $\Cstar$ algebras $\cP_\bullet$ with finite dimensional box spaces $\cP_{n,\pm}$}\,\left\}
\,\,\,\,\cong\,\,
\left\{\,\parbox{9.1cm}{\rm Triples $(\cC, \vee, X)$ with $\cC$ a unitary multitensor category, $\vee$ a unitary dual functor, and a generator $X\in \cC$ with an orthogonal decomposition $1_\cC = 1_+ \oplus 1_-$ such that $X = 1_+\otimes X \otimes 1_-$}\,\right\}.
\right.\right.
\]
Here, we call $X\in \cC$ a generator if every object of $\cC$ is isomorphic to a direct summand 
of alternating tensor powers of $X$ and $X^\vee$.
\end{thmalpha}

As mentioned earlier, for a chosen generator $X\in \cC$ such that $1_\cC = 1_+ \oplus 1_-$ and $X= 1_+\otimes X \otimes 1_-$ as in Theorem \ref{thm:UnitaryPAequivalence}, the canonical standard/minimal solutions to the conjugate equations with respect to $X\in \cC$ from \cite{1805.09234} give a canonical unitary dual functor which makes both loop moduli identical scalars in the corresponding shaded planar $\Cstar$ algebra:
$$
\begin{tikzpicture}[baseline=-.1cm]
	\draw[shaded] (0,0) circle (.3cm);
\end{tikzpicture}
=
d_X
\id_{1_+}
\qquad\qquad
\begin{tikzpicture}[baseline=-.1cm]
	\fill[shaded, rounded corners=5] (-.6,-.6) rectangle (.6,.6);
	\draw[fill=white] (0,0) circle (.3cm);
\end{tikzpicture}
=
d_X 
\id_{1_-}.
$$
There is also a canonical unitary dual functor giving a unitary version of the \emph{lopsided} convention from \cite[\S1.1]{MR3254427} which has been instrumental for constructing many subfactor planar algebras as planar subalgebras of graph planar algebras. 
We refer the reader to \S\ref{sec:ScalarLoops} for more details.

As a final application, in \S\ref{sec:SphericalStates}, we `correct' for some non-spherical choices of unitary pivotal structure on a unitary multitensor category $\cC$.
If $\cC$ is faithfully graded by $\cM_r$, the groupoid with $r$ objects and exactly one isomorphism between any two objects, then any groupoid homomorphism $\pi : \cM_r \to \bbR_{>0}$ induces a unitary pivotal structure on $\cC$ by Theorem \ref{thm:Main} and universality of $\cU$.
As usual, picking $\pi=1$ gives the canonical unitary spherical structure.

\begin{thmalpha}
\label{thm:UniqueSphericalFaithfulState}
Suppose $\dim(\cC(1_\cC \to 1_\cC)) = r$ and $\cC$ is faithfully graded by $\cM_r$.
For each $\pi \in \Hom(\cM_r \to \bbR_{>0})$, there exists a unique faithful state $\psi^\pi : \cC(1_\cC\to 1_\cC) \to \bbC$ such that for all $c\in \cC$ and $f\in \cC(c\to c)$, $\psi^\pi(\tr^\pi_L(f)) = \psi^\pi(\tr^\pi_R(f))$.
\end{thmalpha}

This theorem generalizes the existence of the spherical state on the bipartite graph planar algebra from \cite[Prop.~3.4]{MR1865703} which allows one to construct subfactor planar algebras by finding evaluable planar subalgebras.
There is also a notion of a spherical state with respect to an object $X\in \cC$ from \cite[(7.9)]{1805.09234}; we explain the relation between the two conventions in Example \ref{ex:SphericalFaithfulStateWithRespectToX}.

\paragraph{Acknowledgements.}
The author would like to thank Andr\'{e} Henriques, Corey Jones, and Noah Snyder for many helpful conversations without which this article would not have been possible.
The author would also like to thank Marcel Bischoff, Ian Charlesworth, Sam Evington, Luca Giorgetti, and Cris Negron, a.k.a.~the `Unitary Group' from the 2018 AMS MRC on Quantum Symmetries: Subfactors and Fusion Categories, supported by NSF DMS grant 1641020, for an extremely helpful and clarifying week of discussion on related topics.
The author was supported by NSF DMS grants 1500387/1655912 and 1654159.

\section{Pivotal structures}
\label{sec:PivotalStructures}

In what follows $\cC$ denotes a $\bbC$-linear category.
We write $c\in \cC$ to denote $c$ is an object of $\cC$ and we write $\cC(a\to b)$ for $\Hom_\cC(a,b)$.

\begin{defn}[{\cite[Def.~4.1.1]{MR3242743}}]
A multitensor category is a locally finite $\bbC$-linear abelian rigid monoidal category such that $\otimes: \cC \times \cC \to \cC$ is bilinear.
We call $\cC$ \emph{indecomposable} if it is not equivalent to the direct sum of two nonzero multitensor categories.
If $\cC(1_\cC \to 1_\cC)$ is one-dimensional, i.e., $1_\cC$ is \emph{simple}, we call $\cC$ a tensor category.
\end{defn}

We refer the reader to \cite{MR3242743} for basic background material on multitensor categories.
To ease the notation, whenever possible, we suppress the associator and unitor natural isomorphisms.
All results on pivotal categories in \S\ref{sec:PivotalCategories} -- \ref{sec:PivotalFunctors} are well known to experts.
We provide some proofs for completeness and convenience.

\subsection{Pivotal categories}
\label{sec:PivotalCategories}

We start by recalling the standard definition of a dual functor and a pivotal category.  
For this section, $\cC$ is a rigid monoidal category.
This means for each $c\in \cC$, there exists a dual object $c^\vee\in \cC$ together with evaluation and coevaluation morphisms $\ev_c, \coev_c$ which satisfy the zig-zag axioms, and that for each $c\in \cC$, there is a $c_\vee\in \cC$ such that $(c_\vee)^\vee \cong c$. 

\begin{defn}
A choice of dual $(c^\vee, \ev_c, \coev_c)$ for each $c\in\cC$ assembles into a \emph{dual functor}, which is a strong monoidal functor 
$\vee : \cC \to \cC^{\text{mop}}$\footnote{We use the notation of \cite{1312.7188}; $\cC^{\text{mop}}$ denotes the category obtained from $\cC$ by reversing arrows and reversing the order of tensor product. 
In other words, $\vee: \cC \to \cC$ is contravariant and anti-monoidal.}
defined on $f : \cC(c \to d)$ by
$$
f^\vee 
:=
\begin{tikzpicture}[baseline=-.1cm]
	\draw (0,.3) arc (0:180:.3cm) -- (-.6,-.8);
	\draw (0,-.3) arc (-180:0:.3cm) -- (.6,.8);
	\roundNbox{fill=white}{(0,0)}{.3}{0}{0}{$f$}
	\node at (-.8,-.6) {\scriptsize{$d^\vee$}};
	\node at (.8,.6) {\scriptsize{$c^\vee$}};
\end{tikzpicture}
=
(\ev_d\otimes \id_{c^\vee})
\circ
(\id_{d^\vee}\otimes f\otimes \id_{c^\vee})
\circ
(\id_{d^\vee}\otimes \coev_c).
$$
A dual functor $\vee$ comes with a canonical tensorator $\nu=\{\nu_{a,b}: a^\vee \otimes_{\cC^{\text{mop}}} b^\vee :=b^\vee \otimes a^\vee \to (a\otimes b)^\vee \}$ given by
\begin{equation}
\label{eq:CanonicalTensorator}
\begin{split}
\nu_{a,b} &:=
\begin{tikzpicture}[baseline=-.1cm]
	\draw (.8,.6) -- (.8,0) arc (0:-180:.4cm) arc (0:180:.2cm) -- (-.4,-.6);
	\draw (.7,.6) -- (.7,0) arc (0:-180:.3cm) arc (0:180:.5cm) -- (-.9,-.6);
	\node at (1.4,.4) {\scriptsize{$(a\otimes b)^\vee$}};
	\node at (-.6,-.4) {\scriptsize{$a^\vee$}};
	\node at (-1.1,-.4) {\scriptsize{$b^\vee$}};
\end{tikzpicture}
\\&=
(\ev_b\otimes \id_{(a\otimes b)^\vee})
\circ 
(\id_{b^\vee}\otimes \ev_a \otimes \id_b \otimes \id_{(a\otimes b)^\vee})
\circ
(\id_{b^\vee\otimes a^\vee} \otimes \coev_{a\otimes b})
\end{split}
\end{equation}
and unit isomorphism $r:=\coev_{1_\cC}: 1 \to 1^\vee$.
In what follows, we suppress $r$ to ease the notation.
\end{defn}

\begin{remark}
\label{rem:CoevDeterminesEv}
Given a dual $(c^\vee, \ev_c, \coev_c)$ of $c\in \cC$, the morphism $\ev_c$ is completely determined by $\coev_c$.
(Similarly, $\ev_c$ completely determines $\coev_c$.)
Hence if $(c^\vee_i, \ev_c^i, \coev_c^i)$ for $i=1,2$ are two duals of $c$ and if there is a $\zeta_c\in \cC(c_2^\vee \to c_1^\vee)$ such that $(\id_{c}\otimes \zeta_c)\circ \coev_c^2 = \coev_c^1$, then
\begin{equation}
\label{eq:CanonicalIntertwiner}
\zeta_c
=
\begin{tikzpicture}[baseline=-.1cm]
	\draw[thick, red] (-.4,-.4) -- (-.4,0) arc (180:0:.2cm);
	\draw[thick] (.4,.4) -- (.4,0) arc (0:-180:.2cm);
	\node at (-.6,-.2) {\scriptsize{$c_2^\vee$}};
	\node at (.6,.2) {\scriptsize{$c_1^\vee$}};
\end{tikzpicture}
= 
(\ev_c^2 \otimes \id_{c_1^\vee})\circ (\id_{c^\vee_2} \otimes \coev_c^1)
\end{equation}
which is necessarily invertible, and $\ev_c^1 \circ (\zeta_c\otimes \id_c) = \ev_c^2$ as well.
Hence any two choices of dual functor are \emph{uniquely} monoidally naturally isomorphic.

Moreover, given a dual functor $\vee$, the tensorator $\nu$ from \eqref{eq:CanonicalTensorator} above is \emph{not} part of the data of $\vee$, as it is the unique isomorphism $\zeta_{a\otimes b}$ for the two duals
$((a\otimes b)^\vee, \ev_{a\otimes b}, \coev_{a\otimes b})$ 
and 
$(
b^\vee\otimes a^\vee
,
\ev_b \circ (\id_{b^\vee}\otimes \ev_a\otimes \id_b)
, 
(\id_a \otimes \coev_b\otimes \id_{a^\vee})\circ\coev_a
)$.
\end{remark}

\begin{defn}
A \emph{pivotal structure} on a rigid monoidal category $\cC$ is a pair $(\vee, \varphi)$ where $\vee: \cC \to \cC^{\rm mop}$ is a dual functor and $\varphi: \id\Rightarrow \vee\circ \vee$ is a monoidal natural isomorphism.
This means $\varphi = \{\varphi_c : c \to c^{\vee\vee}\}$ is a collection of natural isomorphisms such that for all $a,b\in \cC$,
the following diagram commutes:
\begin{equation}
\label{eq:EquivalentPivotalStructures}
\begin{tikzcd}
a\otimes b \ar[rr, "\varphi_a\otimes\varphi_b"] \ar[d, "\varphi_{a\otimes b}"]  
&& 
(a\otimes b)^{\vee\vee}
\ar[d, "(\nu_{b^\vee, a^\vee})^\vee"]
\\
a^{\vee\vee} \otimes b^{\vee\vee}
\ar[rr, "\nu_{a^\vee, b^\vee}"]
&& 
(b^\vee \otimes a^\vee)^\vee
\end{tikzcd}
\end{equation}
A \emph{pivotal category} is a rigid monoidal category equipped with a pivotal structure.

Two pivotal structures $(\vee_1,\varphi^1)$ and $(\vee_2, \varphi^2)$ on $\cC$ are \emph{equivalent} if for every $c\in \cC$, the following diagram commutes:
\begin{equation}
\label{eq:EquivalentPivotalStructures}
\begin{tikzcd}
c \ar[rr, "\varphi^1_c"] \ar[drr, "\varphi^2_c"]  
&& 
c_1^{\vee\vee}
\\
&& c_2^{\vee\vee}\ar[u, "{}"]
\end{tikzcd}
\begin{tikzpicture}[baseline=-.1cm]
	\draw[thick, red] (-.4,-.4) -- (-.4,0) arc (180:0:.3cm) arc (0:-180:.1cm);
	\draw[thick] (.4,.4) -- (.4,0) arc (0:-180:.3cm) arc (180:0:.1cm);
	\node at (-.7,-.2) {\scriptsize{$c_2^{\vee\vee}$}};
	\node at (.7,.2) {\scriptsize{$c_1^{\vee\vee}$}};
\end{tikzpicture}
\end{equation}
\end{defn}

\begin{rem}
\label{rem:ClassifyingPivotalStructures}
If $\cC$ has a pivotal structure, then the equivalence classes of pivotal structures on $\cC$ form a torsor for the group $\Aut_\otimes(\id_\cC)$ of monoidal natural automorphisms of the identity functor of $\cC$ \cite[Ex.~4.7.16]{MR3242743}.
\end{rem}

\begin{defn}
Given a pivotal category $(\cC,\varphi)$, we define the left and right trace on $\cC(c\to c)$ for each $c\in \cC$ by
\begin{equation}
\label{eq:DiagrammaticPivotalTraces}
\tr_L^\varphi(f) 
:=
\begin{tikzpicture}[baseline=-.1cm]
	\draw (0,.85) arc (0:180:.3cm) -- (-.6,-.85) arc (-180:0:.3cm) -- (0,.85);
	\roundNbox{fill=white}{(0,.5)}{.35}{0}{0}{$f$}
	\roundNbox{fill=white}{(0,-.5)}{.35}{0}{0}{$\varphi_c^{-1}$}
	\node at (-.9,0) {\scriptsize{$c^\vee$}};
	\node at (.2,0) {\scriptsize{$c$}};
	\node at (.2,1) {\scriptsize{$c$}};
	\node at (.3,-1) {\scriptsize{$c^{\vee\vee}$}};
\end{tikzpicture}
\qquad\qquad
\tr_R^\varphi(f) 
:=
\begin{tikzpicture}[baseline=-.1cm]
	\draw (0,.85) arc (180:0:.3cm) -- (.6,-.85) arc (0:-180:.3cm) -- (0,.85);
	\roundNbox{fill=white}{(0,.5)}{.35}{0}{0}{$\varphi_c$}
	\roundNbox{fill=white}{(0,-.5)}{.35}{0}{0}{$f$}
	\node at (.9,0) {\scriptsize{$c^\vee$}};
	\node at (-.2,0) {\scriptsize{$c$}};
	\node at (-.2,-1) {\scriptsize{$c$}};
	\node at (-.3,1) {\scriptsize{$c^{\vee\vee}$}};
\end{tikzpicture}
\,.
\end{equation}
\end{defn}

\subsection{Semisimple pivotal categories}
\label{sec:SemisimplePivotal}

For this section, $(\cC, \varphi)$ is a semisimple pivotal multitensor category.
This means $\cC(1_\cC \to 1_\cC)$ is a finite dimensional complex semisimple algebra, and is thus isomorphic to $\bbC^r$ for some $r\in \bbN$.
The next lemma is well known to experts; we include a proof for convenience and completeness.

\begin{lem}
\label{lem:Nondegenerate}
The traces $\tr_L^\varphi$ and $\tr_R^\varphi$ are nondegenerate, i.e., for every nonzero $f\in \cC(a\to b)$, there is a $g\in \cC(b\to a)$ such that $\tr_L(g\circ f) \neq 0$, and similarly for $\tr_R$.
\end{lem}
\begin{proof}
Suppose $f\in \cC(a\to b)$ is nonzero.
Then there is a simple $c\in \cC$, a monomorphism $g \in \cC(c\to a)$, and an epimorphism $h\in \cC(b \to c)$ such that $h\circ f\circ g \neq 0$.
Then 
$$
0\neq e := \ev_c \circ [\id_{c^\vee}\otimes (h\circ f\circ g \circ  \varphi_c^{-1})]
\in 
\cC(1 \to c^\vee \otimes c^{\vee\vee}).
$$
Since we also know $0\neq \coev_{c^\vee} \in \cC(c^\vee \otimes c^{\vee\vee})$, 
by \cite[Lem.~A.5]{MR3578212}, 
$$
\tr_L((g\circ h)\circ f) = \tr_L(h\circ f\circ g) = e \circ \coev_{c^\vee} \neq 0.
$$
Hence $\tr_L$ is nondegenerate.
The proof that $\tr_R$ is nondegenerate is similar and left to the reader.
\end{proof}

\begin{defn}
Let $1_\cC = \bigoplus_{i=1}^r 1_i$ be a decomposition into simples, and for $1\leq i\leq r$, let $p_i \in \cC(1_\cC \to 1_\cC)$ be the minimal idempotent corresponding to $1_i$.
For $c\in \cC$ and $f\in \cC(c\to c)$, we define the $M_r(\bbC)$-valued traces $\Tr^\varphi_L$ and $\Tr^\varphi_R$ by the formulas
\begin{equation}
\label{eq:SphericalDefinition}
\begin{split}
(\Tr^\varphi_L(f))_{i,j}
\id_{1_j}
&=
\tr_L^\varphi(p_i \otimes f \otimes p_j) 
=
\begin{tikzpicture}[baseline=-.1cm]
	\draw (0,1) arc (0:180:.5cm) -- (-1,-1) arc (-180:0:.5cm) -- (0,1);
	\roundNbox{fill=white}{(0,.6)}{.35}{0}{0}{$f$}
	\roundNbox{fill=white}{(0,-.6)}{.35}{0}{0}{$\varphi_c^{-1}$}
	\roundNbox{fill=white}{(-.6,0)}{.25}{0}{0}{$p_i$}
	\roundNbox{fill=white}{(.6,0)}{.25}{0}{0}{$p_j$}
	\node at (-1.3,0) {\scriptsize{$c^\vee$}};
	\node at (.15,0) {\scriptsize{$c$}};
	\node at (.15,1.2) {\scriptsize{$c$}};
	\node at (.3,-1.2) {\scriptsize{$c^{\vee\vee}$}};
\end{tikzpicture}
\\
(\Tr^\varphi_R(f))_{i,j}
\id_{1_i}
&=
\tr_R^\varphi(p_i \otimes f \otimes p_j)
=
\begin{tikzpicture}[baseline=-.1cm]
	\draw (0,1) arc (180:0:.5cm) -- (1,-1) arc (0:-180:.5cm) -- (0,1);
	\roundNbox{fill=white}{(0,.6)}{.35}{0}{0}{$\varphi_c$}
	\roundNbox{fill=white}{(0,-.6)}{.35}{0}{0}{$f$}
	\roundNbox{fill=white}{(-.6,0)}{.25}{0}{0}{$p_i$}
	\roundNbox{fill=white}{(.6,0)}{.25}{0}{0}{$p_j$}
	\node at (1.3,0) {\scriptsize{$c^\vee$}};
	\node at (-.15,0) {\scriptsize{$c$}};
	\node at (-.15,-1.2) {\scriptsize{$c$}};
	\node at (-.3,1.2) {\scriptsize{$c^{\vee\vee}$}};
\end{tikzpicture}
\end{split}
\end{equation}
Notice that $\Tr^\varphi_L(f)^T = \Tr^\varphi_R(f^\vee)$ and $\Tr^\varphi_L(f^\vee) = \Tr^\varphi_R(c)^T$ for all $c\in \cC$ and $f\in \cC(c\to c)$.
Moreover, $\Tr^\varphi_L, \Tr^\varphi_R : \cC(c\to c) \to M_r(\bbC)$ are tracial; for all $f\in \cC(c\to d)$ and $g\in \cC(d\to c)$, we have $\Tr^\varphi_L(g\circ f) = \Tr^\varphi_L(f\circ g)$, and similarly for $\Tr^\varphi_R$.
%

We call $(\cC,\varphi)$ \emph{spherical} if for every $c\in\cC$ and $f\in \End_\cC(c)$, $\Tr^\varphi_L(f) = \Tr^\varphi_R(f)$. 

For each $c\in \cC$, we define
$\Dim^\varphi_L(c),\Dim^\varphi_R(c) \in M_r(\bbC)$ by
\begin{equation}
\label{eq:DimensionsInMultitensor}
\Dim^\varphi_L(c) := \Tr^\varphi_L(\id_c)
\qquad
\qquad
\Dim^\varphi_R(c) := \Tr^\varphi_R(c).
\end{equation}
%
Notice that $\Dim^\varphi_L(c)^T = \Dim^\varphi_R(c^\vee)$ and $\Dim^\varphi_L(c^\vee) = \Dim^\varphi_R(c)^T$ for all $c\in \cC$.
Moreover, $\Dim^\varphi_L, \Dim^\varphi_R : K_0(\cC) \to M_r(\bbC)$ are ring homomorphisms.
For each simple $c\in \cC$, the matrices $\Dim^\varphi_L(c)$ and $\Dim^\varphi_R(c)$ have exactly one non-zero entry, which we denote $\dim^\varphi_L(c)$ and $\dim^\varphi_R(c)$ respectively.
\end{defn}

\begin{cor}[{\cite[Prop.~4.8.4]{MR3242743}}]
\label{cor:NonzeroDimensionsForSimples}
For all simple $c\in \cC$, $\dim^\varphi_L(c)\neq 0 \neq \dim^\varphi_R(c)$.
\end{cor}

\begin{defn}
A pivotal structure $(\vee,\varphi)$ on a semisimple multitensor category $\cC$ is called \emph{pseudounitary} if $\dim^\varphi_L(c)>0$ and $\dim^\varphi_R(c) >0$ for all simple $c\in \Irr(\cC)$.
This definition is equivalent to \cite[Def.~9.4.4]{MR3242743} in the context of fusion categories by uniqueness of the Frobenius-Perron dimensions.
\end{defn}

\begin{rem}
\label{rem:ClassifyingPseudounitaryPivotalStructures}
Suppose $\cC$ is a semisimple multitensor category which has a pseudounitary pivotal structure.
Then similar to Remark \ref{rem:ClassifyingPivotalStructures}, the equivalence classes of pseudounitary  pivotal structures on $\cC$ forms a torsor for the subgroup $\Aut^+_\otimes(\id_\cC)$ of $\Aut_\otimes(\id_\cC)$ of \emph{positive} monoidal natural automorphisms of the identity dagger tensor functor, which consists of those monoidal natural isomorphisms $\zeta: \id_\cC \Rightarrow \id_\cC$ such that for every simple $c\in \cC$, $\zeta_c : c\to c$ is a \emph{strictly positve} multiple of $\id_c$.
\end{rem}

\begin{lem}
\label{lem:EquivalentPivotalStructures}
For two pivotal structures $(\vee_1,\varphi^1)$ and $(\vee_2, \varphi^2)$, the following are equivalent:
\begin{enumerate}[(1)]
\item
$(\vee_1,\varphi^1)$ and $(\vee_2, \varphi^2)$ are equivalent.
\item
For all $c \in \cC$ and $f\in \cC(c\to c)$, $\tr_L^1(f) = \tr_L^2(f)$.
\item
For all $c \in \cC$ and $f\in \cC(c\to c)$, $\tr_R^1(f) = \tr_R^2(f)$.
\item
For all simple $c\in \cC$, $\dim_L^1(c) = \dim_L^2(c)$.
\item
For all simple $c\in \cC$, $\dim_R^1(c) = \dim_R^2(c)$.
\end{enumerate}
\end{lem}
\begin{proof}
\mbox{}
\item[\underline{$(1)\Rightarrow (2)$:}]
This is straightforward.
\item[\underline{$(2)\Leftrightarrow (3)$:}]
Note that $\tr_L^1(f) = \tr_R^1(f^\vee)$ and $\tr_L^2(f) = \tr_R^2(f^\vee)$ for all $f\in \cC(c\to c)$.
\item[\underline{$(2)\Rightarrow (4)$:}]
Take $f = \id_c$.
\item[\underline{$(4)\Leftrightarrow (5)$:}]
Note that $\dim_L^1(c) = \dim_R^1(c^\vee_1)$ and $\dim_L^2(c) = \dim_R^2(c_2^\vee)$ for all simple $c\in \cC$.
\item[\underline{$(4)\Rightarrow (1)$:}]
By monoidality of $\varphi^1$, $\varphi^2$, and the canonical intertwining morphism in \eqref{eq:EquivalentPivotalStructures}, $(\vee_1, \varphi^1)$ and $(\vee_2, \varphi^2)$ are equivalent if and only if for all simple $c\in \cC$,
\begin{equation}
\label{eq:EquivalentPivotalIdentity}
(\varphi^2_c)^{-1} \circ
\left(\,
\begin{tikzpicture}[baseline=-.1cm]
	\draw[thick, red] (-.4,-.4) -- (-.4,0) arc (180:0:.3cm) arc (0:-180:.1cm);
	\draw[thick] (.4,.4) -- (.4,0) arc (0:-180:.3cm) arc (180:0:.1cm);
\end{tikzpicture}
\,\right)
\circ
\varphi^1_c
=
\id_c.
\end{equation}
Now the left hand side of \eqref{eq:EquivalentPivotalIdentity} is a scalar multiple of $\id_c$.
By Corollary \ref{cor:NonzeroDimensionsForSimples}, we may determine this scalar by applying $\tr_L^1$ to both sides as $\dim^i_L(c) \neq 0$ for $i=1,2$.
It is straightforward to check that $\tr_L^1$ applied to the left hand side is equal to $\dim^2_L(c)$, which is equal to $\dim^1_L(c)$ by assumption.
Hence \eqref{eq:EquivalentPivotalIdentity} holds.
\end{proof}

\subsection{Pivotal functors}
\label{sec:PivotalFunctors}

\begin{defn}
\label{def:PivotalFunctor}
Given a strong monoidal functor between pivotal categories $(F, \mu) : (\cC, \varphi^\cC) \to (\cD, \varphi^\cD)$, for each $c$ in $\cC$, we have a canonical natural isomorphism $\delta_c : F(c^\vee) \to F(c)^\vee$ given by
\begin{equation}
\label{eq:CanonicalDualIso}
\delta_c
:=
\begin{tikzpicture}[baseline = .3cm]
	\draw (-.5,-1) -- (-.5,0);
	\draw[double] (0,0) -- (0,1.2);
	\draw (.5,-.3) arc (-180:0:.5cm) -- (1.5,1.8);
	\roundNbox{fill=white}{(0,1.2)}{.3}{.5}{.5}{$F(\ev_c)$}
	\roundNbox{fill=white}{(0,0)}{.3}{.7}{.7}{$\mu_{F(c^\vee), F(c)}$}
	\node at (-.9,.6) {\scriptsize{$F(c^\vee\otimes c)$}};
	\node at (-1,-.7) {\scriptsize{$F(c^\vee)$}};
	\node at (.2,-.7) {\scriptsize{$F(c)$}};
	\node at (1.1,.6) {\scriptsize{$F(c)^\vee$}};
\end{tikzpicture}
=
(F(\ev_c) \otimes \id_{F(c)^\vee})\circ (\mu_{F(c^\vee), F(c)}\otimes \id_{F(c)^\vee})\circ (\id_{F(c^\vee)}\otimes \coev_{F(c)})
.
\end{equation}
We call $(F,\mu)$ \emph{pivotal} if for all $c\in \cC$,
\begin{equation}
\label{eqref:PivotalFunctor}
\begin{tikzpicture}[baseline=.6cm]
	\node[box] (F) at (0,0) {$\varphi^\cD_{F(c)}$};
	\node[box] (d) at (0,1.2) {$\delta_{c}^\vee$};
	\draw (F) -- node [left] {\scriptsize $F(c)$} ++(0,-.7);
	\draw (F) -- node [left] {\scriptsize $F(c)^{\vee\vee}$} (d);
	\draw (d) -- node [left] {\scriptsize $F(c^\vee)^\vee$} ++ (0,.7);
\end{tikzpicture}
=
\begin{tikzpicture}[baseline=.6cm]
	\node[box] (F) at (0,0) {$F(\varphi^\cC_c)$};
	\node[box] (d) at (0,1.2) {$\delta_{c^\vee}$};
	\draw (F) -- node [right] {\scriptsize $F(c)$} ++(0,-.7);
	\draw (F) -- node [right] {\scriptsize $F(c^{\vee\vee})$} (d);
	\draw (d) -- node [right] {\scriptsize $F(c^\vee)^\vee$} ++ (0,.7);
\end{tikzpicture}.
\end{equation}
\end{defn}

\begin{lem}
\label{lem:PreserveDimensions}
Suppose $(\cC,\varphi^\cC)$ and $(\cD, \varphi^\cD)$ are pivotal categories 
and 
$(F,\mu): \cC \to \cD$ is a pivotal strong monoidal functor.
Then $(F, \mu)$ preserves the left and right pivotal traces, i.e., for all $f\in \cC(c\to c)$, 
\begin{equation}
\label{eq:PreserveLeftDimension}
F\left(
\begin{tikzpicture}[baseline=-.1cm]
	\draw (0,.85) arc (0:180:.3cm) -- (-.6,-.85) arc (-180:0:.3cm) -- (0,.85);
	\roundNbox{fill=white}{(0,.5)}{.35}{0}{0}{$f$}
	\roundNbox{fill=white}{(0,-.5)}{.35}{0}{0}{$\varphi_c^{-1}$}
	\node at (-.9,0) {\scriptsize{$c^\vee$}};
	\node at (.2,0) {\scriptsize{$c$}};
	\node at (.2,1) {\scriptsize{$c$}};
	\node at (.3,-1) {\scriptsize{$c^{\vee\vee}$}};
\end{tikzpicture}
\right)
=
F(\tr^{\varphi^\cC}_L(f))
=
\tr_L^{\varphi^\cD}(F(f))
=
\begin{tikzpicture}[baseline=-.1cm]
	\draw (0,.85) arc (0:180:.4cm) -- (-.8,-.85) arc (-180:0:.4cm) -- (0,.85);
	\roundNbox{fill=white}{(0,.5)}{.3}{.2}{.2}{$F(f)$}
	\roundNbox{fill=white}{(0,-.5)}{.35}{.2}{.2}{$\varphi_{F(c)}^{-1}$}
	\node at (-1.2,0) {\scriptsize{$F(c)^\vee$}};
	\node at (.3,0) {\scriptsize{$F(c)$}};
	\node at (.3,1) {\scriptsize{$F(c)$}};
	\node at (.5,-1) {\scriptsize{$F(c)^{\vee\vee}$}};
\end{tikzpicture}
\,,
\end{equation}
and similarly for the right pivotal traces.
\end{lem}
\begin{proof}
Notice that for $f\in \cC(c\to c)$, we always have
\begin{equation}
\label{eq:FOfLeftDimension}
F(\tr_L^\cC(f))
=
\begin{tikzpicture}[baseline=-.6cm]
	\draw (-.4,-2) -- (-.4,1);
	\draw (-1.2,-2) -- (-1.2,1);
	\draw[double] (-.8,1) -- (-.8,2);
	\draw[double] (-.8,-2) -- (-.8,-3);
	\roundNbox{fill=white}{(-.8,2)}{.3}{.3}{.3}{$F(\ev_{c})$}
	\roundNbox{fill=white}{(-.8,1)}{.3}{.5}{.5}{$\mu$}
	\roundNbox{fill=white}{(-.4,0)}{.3}{.2}{.2}{$F(f)$}
	\roundNbox{fill=white}{(-.4,-1)}{.35}{.3}{.3}{$F(\varphi_{c}^{-1})$}
	\roundNbox{fill=white}{(-.8,-2)}{.3}{.5}{.5}{$\mu^{-1}$}
	\roundNbox{fill=white}{(-.8,-3)}{.3}{.6}{.6}{$F(\coev_{c^\vee})$}
	\node at (-1.5,1.5) {\scriptsize{$F(c^\vee \otimes c)$}};
	\node at (-.1,.5) {\scriptsize{$F(c)$}};
	\node at (-.1,-.5) {\scriptsize{$F(c)$}};
	\node at (-1.6,-.5) {\scriptsize{$F(c^\vee)$}};
	\node at (.1,-1.5) {\scriptsize{$F(c^{\vee\vee})$}};
	\node at (-1.7,-2.5) {\scriptsize{$F(c^\vee \otimes c^{\vee\vee})$}};
\end{tikzpicture}
=
\begin{tikzpicture}[baseline=-.1cm]
	\draw (-1.4,-1.3) -- (-1.4,1.3) arc (180:0:.5cm) -- (-.4,-1.3) arc (0:-180:.5cm);
	\roundNbox{fill=white}{(-1.4,1)}{.3}{0}{0}{$\delta_c$}
	\roundNbox{fill=white}{(-.4,1)}{.3}{.2}{.2}{$F(f)$}
	\roundNbox{fill=white}{(-.4,0)}{.35}{.3}{.3}{$F(\varphi_{c}^{-1})$}
	\roundNbox{fill=white}{(-.4,-1)}{.3}{.1}{.1}{$\delta_{c^\vee}^{-1}$}
	\node at (-1.8,1.5) {\scriptsize{$F(c)^\vee$}};
	\node at (-.1,1.5) {\scriptsize{$F(c)$}};
	\node at (-.1,.5) {\scriptsize{$F(c)$}};
	\node at (-1.8,0) {\scriptsize{$F(c^\vee)$}};
	\node at (.1,-.5) {\scriptsize{$F(c^{\vee\vee})$}};
	\node at (.1,-1.5) {\scriptsize{$F(c^\vee)^{\vee}$}};
\end{tikzpicture}
=
\begin{tikzpicture}[baseline=-.1cm]
	\draw (-1.4,-2.3) -- (-1.4,1.3) arc (180:0:.5cm) -- (-.4,-2.3) arc (0:-180:.5cm);
	\roundNbox{fill=white}{(-.4,-2)}{.3}{0}{0}{$\delta_c^\vee$}
	\roundNbox{fill=white}{(-.4,1)}{.3}{.2}{.2}{$F(f)$}
	\roundNbox{fill=white}{(-.4,0)}{.35}{.3}{.3}{$F(\varphi_{c}^{-1})$}
	\roundNbox{fill=white}{(-.4,-1)}{.3}{.1}{.1}{$\delta_{c^\vee}^{-1}$}
	\node at (-1.8,1.5) {\scriptsize{$F(c)^\vee$}};
	\node at (-.1,1.5) {\scriptsize{$F(c)$}};
	\node at (-.1,.5) {\scriptsize{$F(c)$}};
	\node at (-1.8,0) {\scriptsize{$F(c^\vee)$}};
	\node at (.1,-.5) {\scriptsize{$F(c^{\vee\vee})$}};
	\node at (.1,-1.5) {\scriptsize{$F(c^\vee)^{\vee}$}};
	\node at (.1,-2.5) {\scriptsize{$F(c)^{\vee\vee}$}};
\end{tikzpicture}
\end{equation}
If $(F,\mu)$ is pivotal, then the right hand sides of \eqref{eq:PreserveLeftDimension} and \eqref{eq:FOfLeftDimension} above are equal.
The proof for the right pivotal trace is analogous.
\end{proof}

The converse of Lemma \ref{lem:PreserveDimensions} is true under some additional assumptions.

\begin{lem}
If $(\cC,\varphi^\cC),(\cD,\varphi^\cD)$ are pivotal semisimple multitensor categories 
and $(F,\mu): \cC \to \cD$ is a full strong monoidal functor which preserves the left or right pivotal traces, then $(F,\mu)$ is fully faithful and pivotal.
\end{lem}
\begin{proof}
We assume $(F,\mu)$ preserves the left pivotal traces, and the proof for the right pivotal traces is analogous.
First, suppose $f\in \cC(c\to d)$.
By nondegeneracy of $\tr^{\varphi^\cC}_L$, there is a $g\in \cC(d\to c)$ such that $\tr^{\varphi^\cC}_L(g\circ f) \neq 0$.
Then $\tr^{\varphi^\cD}_L(F(g)\circ F(f)) \neq 0$, so $F(f) \neq 0$ and $F$ is fully faithful.
Notice this immediately implies $F$ takes simples to simples, and non-isomorphic simples in $\cC$ remain non-isomorphic in $\cD$.
Now to show $(F, \mu)$ is pivotal, by monoidality, it suffices to prove \eqref{eqref:PivotalFunctor} when $c\in \cC$ is simple.
By the above argument, $F(c)$ is then simple, and every morphism in $\cD(F(c) \to F(c))$ is a scalar multiple of $\id_{F(c)}$.
Since the right hand side of \eqref{eq:PreserveLeftDimension} is equal to the right hand side of \eqref{eq:FOfLeftDimension} by hypothesis, by nondegeneracy of the trace from Lemma \ref{lem:Nondegenerate}, we must have
$\varphi^{-1}_{F(c)} = F(\varphi^{-1}_c) \circ \delta^{-1}_{c^\vee}\circ \delta^\vee_c$, and thus \eqref{eqref:PivotalFunctor} holds.
\end{proof}

\section{Unitary dual functors}

We begin this section with background on dagger structures and $\Cstar$ categories in \S\ref{sec:DaggerStructures}.
Next, we give the correct notion of unitary dual functor and unitary pivotal structure from \cite[\S7.3]{MR2767048}.
Similar to the situation for tensor categories, in \S\ref{sec:GradingGroupoid}, we classify pivotal structures on a multitensor category via homomorphisms out of the universal grading groupoid $\cU$.
In \S\ref{sec:BalancedDuals}, for a unitary multitensor category, we construct a canonical unitary dual functor from each groupoid homomorphism $\pi \in \Hom(\cU \to \bbR_{>0})$, and we show these exhaust the possible unitary equivalence classes of unitary dual functors.
We then describe the canonical bi-involutive structure associated to a unitary dual functor in \S\ref{sec:Bi-involutive}.

\subsection{Dagger structures and unitary multitensor categories}
\label{sec:DaggerStructures}

\begin{defn}
Given a $\bbC$-linear category $\cC$, a \emph{dagger structure} is a collection of anti-linear maps $\dag : \cC(c\to d) \to \cC(d\to c)$ for all $c,d\in\cC$ such that $(f\circ g)^\dag = g^\dag \circ f^\dag$ for composable $f$ and $g$, and $f^{\dag\dag} = f$ for all $f$.
A morphism $f : \cC(a \to b)$ is called \emph{unitary} if $f^\dag = f^{-1}$. 

A \emph{dagger (multi)tensor category} is a (multi)tensor category equipped with a dagger
structure so that $(f \otimes g)^\dag = f^\dag \otimes g^\dag$ for all
morphisms $f, g$, and all associator and unitors are unitary natural isomorphisms.
\end{defn}

\begin{defn}
A functor between dagger categories $F: \cM \to \cN$ is called a \emph{dagger functor} if $F(f^\dag) = F(f)^\dag$ for all morphisms $f$ in $\cM$.
Given finitely semisimple dagger categories $\cM$ and $\cN$, we define $\Fun^\dag(\cM\to \cN)$ to be the dagger category of dagger functors from $\cC \to \cD$ with dagger structure defined as follows.
Given a natural transformation of dagger functors $\eta: F \Rightarrow G$, it is straightforward to show that $(\eta^\dag)_m := (\eta_m)^\dag$ for $m\in\cM$ gives a well-defined natural transformation $\eta^\dag:G\Rightarrow F$.\footnote{In the $\rm C^*$ setting, one only considers \emph{bounded} natural transformations, i.e., those for which $\sup_{m \in \cM} \|\eta_m\| <\infty$.}
One now calculates that $\eta \mapsto \eta^\dag$ defines a dagger structure on $\Fun^\dag(\cM\to \cN)$.
It is important to note that a natural transformation $\eta: F \Rightarrow G$ is unitary if and only if $\eta_m \in \cN(F(m) \to G(m))$ is unitary for all $m\in\cM$.

A \emph{dagger equivalence} of dagger categories $\cM$ and $\cN$ consists of dagger functors $F: \cM \to \cN$ and $G: \cN \to \cM$ together with unitary natural isomorphisms $\id \Rightarrow F\circ G$ and $\id \Rightarrow G\circ F$.
A tensor functor between dagger tensor categories $(F, \mu): \cC \to \cD$ is called a \emph{dagger tensor functor} if $F$ is a dagger functor and $\mu_{c,d}$ is unitary for all $c,d\in \cC$.
Given a finitely semisimple dagger category $\cM$, $\End^\dag(\cM):= \Fun^\dag(\cM,\cM)$ is easily seen to be a strict semisimple dagger multitensor category.
\end{defn}

\begin{remark}
\label{rem:PrincipleOfEquivalence}
The \emph{principle of equivalence} \cite{nlab:principle_of_equivalence} in category theory roughly states that a properly defined structure should be invariant under the proper notion of equivalence.
The proper notion of equivalence between two objects in a dagger category is that they are \emph{unitary isomorphic}, and the proper notion of equivalence between dagger functors between dagger categories is that they are \emph{unitarily naturally isomorphic}. 
For example, if $F,G: \cC \to \cD$ are functors between dagger categories with $F$ a dagger functor, and $\eta: F\Rightarrow G$ is a natural isomorphism, $G$ need not be a dagger functor unless $\eta$ is unitary.

With this in mind, a dagger category \emph{cannot} be considered as a category with some extra categorical structure.
For example, if $\cD$ is a dagger category and the underlying category of $\cD$ is equivalent to the category $\cC$, there is generally no way to endow $\cC$ with a dagger structure which promotes the equivalence to a dagger equivalence.
We refer the reader to the helpful discussion between Shulman and Selinger available at \cite{nlab:principle_of_equivalence} for further details.
\end{remark}

\begin{remark}
The forgetful functor on a finitely semisimple dagger (multitensor) category which forgets the dagger structure is a fully faithful (tensor) functor.
Thus 
the category $\fdHilb$ of finite dimensional Hilbert spaces 
is equivalent to the 
category $\fdVec$ 
of finite dimensional vector spaces
as a linear (tensor) category.
Also, for a finitely semisimple dagger category $\cM$, $\End^\dag(\cM)\cong \End(\cM)$ as linear multitensor categories.
\end{remark}

\begin{defn}
A dagger category which admits orthogonal direct sums is called a \emph{$\Cstar$ category} if every endomorphism algebra is a $\Cstar$-algebra (see \cite{MR808930,MR3663592}).
Notice that a finitely semisimple dagger category is $\rm C^*$ if and only if it is dagger equivalent to $\Hilb^n$ for some $n\in \bbN$.
\end{defn}

\begin{defn}
A \emph{tensor $\Cstar$ category} is a linear monoidal dagger category which admits orthogonal direct sums and is idempotent complete, and whose underlying dagger category is $\Cstar$.
A \emph{unitary (multi)tensor category} is a semisimple (multi)tensor $\Cstar$ category.\footnote{
By a generalization of \cite[Lem.~3.9]{MR1444286} (see also the second paragraph on p.~9 therein), a tensor $\Cstar$ category is rigid if and only if it is semisimple.
Hence the adjective `tensor' in `tensor $\Cstar$ category' does \emph{not} include `rigid' nor having simple unit object, in conflict with the definition of `tensor category' following \cite{MR3242743}.
}
As before, the prefix \emph{multi-} is used if and only if $1_\cC$ is not simple.

A \emph{unitary (multi)fusion category} is a finitely semisimple unitary (multi)tensor category.
For $r>1$, an \emph{$r\times r$ unitary multifusion category} is an indecomposable unitary multifusion category such that $\dim(\cC(1_\cC \to 1_\cC)) = r$, so we can orthogonally decompose $1_\cC$ into simples as $\bigoplus_{i=1}^r 1_i$.
\end{defn}

\begin{warn}
\label{warn:Pseudounitary}
While natural from a \emph{non-unitary} viewpoint, pseudounitary pivotal structures are \emph{unnatural} for unitary multitensor categories.
The problem arises from the fact that while a dual functor is unique up to unique natural isomorphism as in \eqref{eq:CanonicalIntertwiner}, this unique natural isomorphism need not be unitary!
Hence one may only discuss the compatibility of a \emph{fixed} dual functor $\vee: \cC \to \cC^{\rm mop}$ with $\dag$.
However, compatibility of $\dag$ with $\vee$ need not imply compatibility of $\dag$ with an equivalent dual functor $\vee': \cC \to \cC^{\rm mop}$.
We provide Lemma \ref{lem:DagVeeProblem} below which gives a sufficient condition to transport the compatibility.
In \S\ref{sec:BalancedDuals} below, we give a \emph{manifestly unitary} approach, and we reconcile the latter with the former.
\end{warn}

\begin{lem}
Suppose $\cC$ is a unitary multitensor category and $\varphi$ is a pivotal structure.
The following are equivalent.
\begin{enumerate}[(1)]
\item
$(\cC, \varphi)$ is pseudounitary.
\item
For all $a,b\in \cC$ and all $f\in \cC(a\to b)$ with $f\neq 0$, 
$\tr_L(f^\dag \circ f) > 0$ and $\tr_R(f^\dag \circ f) > 0$.\footnote{
Note that if $f^{\vee\dag} = f^{\dag\vee}$, then $\tr_L(f^\dag \circ f) > 0$ if and only if $\tr_R(f^\dag \circ f) > 0$.
}
\end{enumerate}
\end{lem}
\begin{proof}
That $(2)\Rightarrow (1)$ is trivial.
Suppose $(\cC, \varphi)$ is pseudounitary.
We show that for all $a,b\in \cC$ and $f\in \cC(a\to b)$ with $f\neq 0$, $\tr_L^\varphi(f^\dag\circ f)>0$.
The proof that $\tr^\varphi_R(f\circ f^\dag)>0$ is similar.

\item[\underline{Step 1:}]
Suppose $c\in \cC$ is simple and $f\in \cC(c\to c)$ with $f\neq 0$.
Then $f = \lambda \id_c$ for some $\lambda \in \bbC^\times$, so $f^\dag\circ f = |\lambda|^2\id_c$, and $\tr_L^\varphi(f^\dag\circ f) = |\lambda|^2\dim^\varphi_L(c) >0$.

\item[\underline{Step 2:}]
Suppose $a,b\in \cC$ are respectively orthogonal direct sums of $m,n$ objects isomorphic to the simple object $c\in \cC$ and $f\in \cC(a\to b)$ with $f\neq 0$.
Pick $m$ isometries $v_1,\dots, v_m\in \cC(c\to a)$ so that $\sum_{i=1}^m v_i \circ v_i^\dag = \id_a$.
Note that $f\neq 0$ if and only if $v_i^\dag \circ f^\dag \circ f \circ v_i \in \cC(c\to c)$ is nonzero for some $i=1,\dots, m$.
Thus by Step 1,
$$
\tr^\varphi_L(f^\dag \circ f)
=
\sum_{i=1}^m 
\tr^\varphi_L(v_i \circ v_i^\dag\circ f^\dag \circ f)
=
\sum_{i=1}^m 
\tr^\varphi_L(v_i^\dag\circ f^\dag \circ f \circ v_i)
>0.
$$

\item[\underline{Step 3:}]
For arbitrary $a,b\in \cC$ and $f\in \cC(a\to b)$ nonzero, decompose $a$ and $b$ into orthogonal direct sums of isotypic components and apply Step 2.
\end{proof}

\subsection{Unitary dual functors}

For this section, $\cC$ is a unitary multitensor category

The following proposition is \cite[Lem.~7.5]{MR2767048}, which can be viewed as a generalization of \cite[Lem.~2.16]{MR2794547} in the non-strict unitary multitensor category setting.

\begin{prop}
\label{prop:VeeAndDagCompatibility}
Fix a dual functor $\vee : \cC \to \cC^{\rm mop}$ with its canonical tensorator $\nu$ from \eqref{eq:CanonicalTensorator}.
The following are equivalent.
\begin{enumerate}[(1)]
\item
$\vee$ is a dagger tensor functor, i.e., 
for all $a,b\in \cC$ and $f\in \cC(a\to b)$, $\nu_{a,b}$ is unitary and $f^{\vee\dag} = f^{\dag\vee}$.
\item
Defining $\varphi_c:=(\coev_{c}^\dag \otimes \id_{c^{\vee\vee}})\circ (\id_c\otimes \coev_{c^\vee})$ gives a pivotal structure $\varphi: \id \Rightarrow \vee\circ\vee$.
\end{enumerate}
\end{prop}
\begin{proof}
First, note that $\varphi$ is natural if and only if 
$$
\varphi_b\circ f =f^{\vee\vee}\circ \varphi_a
\qquad
\forall f\in \cC(a\to b)
$$
if and only if 
$$
\coev_b^\dag\circ (f \otimes \id_{b^\vee})
=
\coev_a^\dag \circ (\id_{a}\otimes f^\vee)
\qquad
\forall f\in \cC(a\to b)
$$
if and only if $f^{\dag\vee}=f^{\vee\dag}$ for all $f\in \cC(a\to b)$.
Second, note that $\varphi$ is monoidal if and only if 
$$
\coev_a^\dag\circ (\id_a \otimes \coev^\dag_b \otimes \id_{a^\vee})
=
\coev^\dag_{a\otimes b} \circ (\id_{a\otimes b} \otimes \nu_{b,a})
\qquad
\forall a,b\in \cC
$$
if and only if 
$$
\id_{b^\vee\otimes a^\vee}
=
(\id_{b^\vee \otimes a^\vee} \otimes \coev_{a\otimes b}^\dag)
\circ
(\id_{b^\vee} \otimes \ev_a^\dag \otimes \id_b \otimes \id_{(a\otimes b)^\vee})
\circ
(\ev_b^\dag \otimes \id_{(a\otimes b)^\vee})
\circ
\nu_{b,a}
\qquad
\forall a,b\in\cC
$$
if and only if $\nu_{a,b}$ is unitary for all $a,b\in \cC$.
\end{proof}

\begin{cor}
\label{cor:UnitaryPivotal}
If either of the equivalent conditions of Proposition \ref{prop:VeeAndDagCompatibility} hold, then for all $c\in \cC$, 
\begin{equation}
\label{eq:UnitaryPivotalStructure}
(\coev_{c}^\dag \otimes \id_{c^{\vee\vee}})\circ (\id_c\otimes \coev_{c^\vee})
=:
\varphi_c
=
(\id_{c^{\vee\vee}} \otimes \ev_c)\circ ( \ev_{c^\vee}^\dag \otimes \id_c),
\end{equation}
which is equivalent to $\varphi_c$ being unitary for all $c\in \cC$. 
\end{cor}
\begin{proof}
By (1) of Proposition \ref{prop:VeeAndDagCompatibility}, we have  $\coev_c^{\dag\vee} = \coev_c^{\vee\dag} = \ev_{c_\vee}^\dag$, which is equivalent to \eqref{eq:UnitaryPivotalStructure}.
Now notice that $\varphi_c^\dag$ is the inverse of the expression on the right hand side of \eqref{eq:UnitaryPivotalStructure}, so \eqref{eq:UnitaryPivotalStructure} holds if and only if $\varphi_c$ is unitary. 
\end{proof}

\begin{defn}
A dual functor $\vee: \cC \to \cC^{\text{mop}}$ is called a \emph{unitary dual functor}
if any of the equivalent conditions of Proposition \ref{prop:VeeAndDagCompatibility} hold.
Two unitary dual functors are called \emph{unitarily equivalent} if the canonical monoidal natural isomorphism from \eqref{eq:CanonicalIntertwiner} is unitary.
\end{defn}

In line with the principle of equivalence for dagger categories discussed in Remark \ref{rem:PrincipleOfEquivalence}, unitary equivalence between a unitary dual functor and an arbitrary dual functor transports unitarity, as we will see right below in Lemma \ref{lem:DagVeeProblem}.
Of course, two unitary dual functors need not be unitarily equivalent, as can be seen from the construction in \S\ref{sec:BalancedDuals} together with Lemma \ref{lem:UnitarilyEquivalentUnitaryDuals} below.

\begin{lem}
\label{lem:DagVeeProblem}
Suppose $\nu_1,\nu_2 : \cC \to \cC^{\rm mop}$ are two dual functors such that $\vee_1$ unitary.
If for all $c\in \cC$, the canonical isomorphism $\zeta_c\in \cC(c^{\vee_2} \to c^{\vee_1})$ from \eqref{eq:CanonicalIntertwiner} is unitary, then $\vee_2$ is unitary.
\end{lem}
\begin{proof}
Suppose that the canonical isomorphism in \eqref{eq:CanonicalIntertwiner} is always unitary.
Then for all $f\in \cC(a\to b)$, 
\begin{align*}
f^{\vee_2\dag}
&=
(\id_{b^{\vee_2}} \otimes (\coev^2_a)^\dag)\circ (\id_{b^{\vee_2}}\otimes f^\dag \otimes \id_{a^{\vee_2}})\circ ((\ev^2_b)^\dag \otimes \id_{a^{\vee_2}})
\\&=
(\id_{b_2^\vee} \otimes (\coev_b^1)^\dag \otimes (\coev_a^1)^\dag\otimes (\coev^2_a)^\dag)
\circ 
(\id_{b^{\vee_2}\otimes b\otimes b^{\vee_1}}\otimes f^\dag \otimes \id_{a^{\vee_1}\otimes a\otimes a^{\vee_2}})
\\&\hspace{6.91cm}\circ 
((\ev^2_b)^\dag \otimes (\ev^1_b)^\dag\otimes (\ev^1_a)^\dag\otimes \id_{a^{\vee_2}})
\\&=
(\ev_a^2 \otimes \id_{a^{\vee_1}})\circ (\id_{a^{\vee_2}} \otimes \coev_a^1)
\circ 
f^{\dag\vee_1}
\circ 
(\ev_b^2 \otimes \id_{b^{\vee_1}})\circ (\id_{b^{\vee_2}} \otimes \coev_b^1)
\\&=
f^{\dag\vee_2}.
\end{align*}
Moreover, for all $a,b\in \cC$, we have
$$
\nu^2_{a,b}
=
\zeta_{b\otimes a}^{-1}
\circ
\nu_{a,b}^1
\circ
(\zeta_a\otimes \zeta_b)
\in 
\cC(a^{\vee_2 }\otimes b^{\vee_2}\to (b\otimes a)^{\vee_2}),
$$
which is necessarily unitary as it is a composite of unitaries.
Hence $\vee_2$ is unitary.
\end{proof}

\begin{defn}
We call a pivotal structure $(\vee,\varphi)$ \emph{unitary} if $\vee$ is a unitary dual functor and $\varphi$ is as in \eqref{eq:UnitaryPivotalStructure}.
Two unitary pivotal structures are \emph{unitarily equivalent} if they are equivalent and the canonical monoidal natural isomorphism from \eqref{eq:EquivalentPivotalStructures} is unitary.
\end{defn}

\begin{rem}
\label{rem:UnitaryImpliesPseudounitary}
For a unitary dual functor $\vee$, the left and right pivotal traces have alternate formulas that show they are manifestly positive linear operators $\cC(c\to c) \to \cC(1_\cC \to 1_\cC)$:
\begin{align*}
\tr_L^\varphi(f) 
&
\underset{\text{\eqref{eq:DiagrammaticPivotalTraces}}}{=}
\ev_c
\circ 
(\id_{c^\vee} \otimes f)
\circ 
(\id_{c^\vee}\otimes \varphi_c^{-1})
\circ
\coev_{c^\vee}
\underset{\text{Cor.~\ref{cor:UnitaryPivotal}}}{=}
\ev_c \circ (\id_{\overline{c}^\vee}\otimes f) \circ \ev_c^\dag
\\
\tr_R^\varphi(f)
&
\underset{\text{\eqref{eq:DiagrammaticPivotalTraces}}}{=}
\ev_c
\circ 
(\varphi_c \otimes \id_{c^\vee})
\circ 
(f\otimes \id_{c^\vee} )
\circ
\coev_{c^\vee}
\underset{\text{Cor.~\ref{cor:UnitaryPivotal}}}{=}
\coev_c^\dag \circ (f\otimes \id_{c^\vee})\circ \coev_c.
\end{align*}
Hence every unitary pivotal structure is pseudounitary.
We will use these alternate formulas in \S\ref{sec:BalancedDuals} below.
\end{rem}

\begin{lem}
\label{lem:UnitarilyEquivalentUnitaryDuals}
Fix two unitary dual functors $\vee_1, \vee_2$, and let $\varphi^1,\varphi^2$ be the respective induced unitary pivotal structures.
We have $\vee_1$ and $\vee_2$ are unitarily equivalent if and only if $\varphi^1$ and $\varphi^2$ are unitarily equivalent.
\end{lem}
\begin{proof}
Recall that for $c\in \cC$, $\zeta_c\in \cC(c^{\vee_2} \to c^{\vee_1})$ is the unique natural isomorphism from \eqref{eq:CanonicalIntertwiner}.
Observe that
$\vee_1$ and $\vee_2$ are unitarily equivalent
if and only if
$$
\zeta_c^{-1}
=
(\id_{c^{\vee_2}} \otimes \coev_1^\dag)
\circ
(\ev_2^\dag \otimes \id_{c^{\vee_1}})
=
\zeta_c^\dag
\qquad
\forall c\in \cC
$$
if and only if
$$
\begin{tikzpicture}[baseline=-.1cm]
	\draw[thick, red] (-.4,-.4) -- (-.4,0) arc (180:0:.3cm) arc (0:-180:.1cm);
	\draw[thick] (.4,.4) -- (.4,0) arc (0:-180:.3cm) arc (180:0:.1cm);
\end{tikzpicture}
=
\varphi^2_c\circ (\varphi^1_c)^{-1}
=
(\coev_1^\dag \otimes \id_{c^{\vee_1\vee_1}})
\circ
(\id_c \otimes \coev_{c^{\vee_1}})
\circ
(\ev_{c^{\vee_2}} \otimes \id_c)
\circ
(\id_{c^{\vee_2\vee_2}}\otimes \ev_2^\dag)
\qquad
\forall c\in \cC
$$
if and only if $\varphi^1$ and $\varphi^2$ are unitarily equivalent.
\end{proof}

\subsection{The universal grading groupoid}
\label{sec:GradingGroupoid}

We now adapt \cite[\S4.14]{MR3242743} to the (unitary) multitensor category setting.
For this section, $\cC$ is a multitensor category which is not necessarily unitary.
In the unitary setting, one should add the terms in parentheses, and one may ignore them in the algebraic setting.

Recall that a \emph{groupoid} $\cG$ is a category where all morphisms are invertible.
As the groupoids we consider have only finitely many objects, we will identify $\cG$ with its set of morphisms, and we can recover the objects as the \emph{idempotents}, i.e., those morphisms $e\in \cG$ such that $e\circ e = e$.

\begin{defn}
A \emph{grading} of $\cC$ by a groupid $\cG$ is a decomposition 
$$
\cC= \bigoplus_{g\in \cG} \cC_g
$$
where each $\cC_g\subset \cC$ is a semisimple ($\Cstar$) subcategory such that if $g,h$ are composable, the tensor product maps $\cC_g \times \cC_h$ to $\cC_{gh}$.
When $g,h$ are not composable, the tensor product on $\cC_g \times \cC_h$ is the zero bi-functor.
For every idempotent in $e\in\cG$, $\cC_e$ is a (unitary) multitensor subcategory of $\cC$.
A grading by $\cG$ is called \emph{faithful} if $\cC_g \neq 0$ for all $g\in \cG$.
Gradings are in bijection with gradings of the Grothendieck ring $K_0(\cC)$ as a based ring, where the basis corresponds to the isomorphism classes of simple objects of $\cC$.

Given any two faithful gradings, there is a common faithful refinement, so there exists a universal grading of $\cC$.
We call the groupoid associated to the universal grading of $\cC$ the \emph{universal grading groupoid}, denoted $\cU$.
\end{defn}

\begin{rem}
\label{rem:CanonicalGroupoidSurjection}
If $\cC$ is faithfully graded by $\cG$, then since $\cU$ is a refinement of $\cG$, we get canonical surjective groupoid homomorphism $\cU \twoheadrightarrow\cG$ given by mapping a $u \in \cU$ to the $g\in \cG$ such that $\cC_u \subset \cC_g$.
\end{rem}

\begin{nota}
For $c\in \cC$, we say $c$ is \emph{homogeneous} if $c$ lies in one component subcategory $\cC_g\subset \cC$ for some $g\in \cU$.
For such $c$, we define $\gr(c) := g$.
For an arbitrary $c\in \cC$, we write $c = \bigoplus_{g\in \cU} c_g$ for the canonical (orthogonal) decomposition of $c$ into homogeneous subobjects.
For an $f\in \cC(a\to b)$, we write $f_g \in \cC(a_g \to b_g)$ for the $g$-graded component of $f$.

For a groupoid $\cG$ and an abelian group $A$ (whose group law is still denoted multiplicatively), we denote by $\Hom(\cG \to A)$ the set of functors from $\cG$ to $A$ where the latter is viewed as a groupoid with exactly one object.
Note that $\Hom(\cG \to A)$ is a group under pointwise multiplication and pointwise inversion.
\end{nota}

Recall from Remarks \ref{rem:ClassifyingPivotalStructures} and \ref{rem:ClassifyingPseudounitaryPivotalStructures} that
$\Aut_\otimes(\id_\cC)$ 
is the group of of monoidal natural automorphisms of the identity tensor functor, and $\Aut_\otimes^+(\id_\cC)$ is the subgroup of positive monoidal natural isomorphisms of the identity dagger tensor functor.

\begin{lem}
\label{lem:MonoidalNaturalAutomorphisms}
There is a canonical isomorphism $\Aut_\otimes(\id_\cC)\cong \Hom(\cU \to \bbC^\times)$ which takes the subgroup $\Aut^+_\otimes(\id_\cC)$ onto the subgroup $\Hom(\cU \to \bbR_{>0})$.
\end{lem}
\begin{proof}
Given $\zeta \in \Aut_\otimes(\id_\cC)$, we get a grading of $\cC$ by $\bbC^\times$ by assigning to each simple $c\in \cC$ the number corresponding to $\zeta_c \in \cC(c\to c) = \bbC \id_c$.
This gives us a homomorphism $f_\zeta : \cU \to \bbC^\times$ by universality of $\cU$.
One now checks the map $\zeta \mapsto f_\zeta$ is an isomorphism.
Finally, $\zeta \in \Aut^+_\otimes(\id_\cC)$ if and only if $\zeta_c\in \bbR_{>0}\id_c$ for all simple $c\in \cC$ if and only if $\im(f_\zeta) \subset \bbR_{>0}$.
\end{proof}

For the convenience of the reader, we provide the lemma below which, in the presence of a pivotal structure $(\vee, \varphi)$, provides an explicit bijection between the torsor of pivotal structures on $\cC$ and $\Hom(\cU \to \bbC^\times)$ obtained by combining Remark \ref{rem:ClassifyingPivotalStructures} and Lemma \ref{lem:MonoidalNaturalAutomorphisms} above.
Moreover, in the presence of a pseudounitary pivotal structure, this bijection restricts to a bijection between the torsor of pseudounitary pivotal structures and $\Hom(\cU \to \bbR_{>0})$ obtained by combining Remark \ref{rem:ClassifyingPseudounitaryPivotalStructures} with Lemma \ref{lem:MonoidalNaturalAutomorphisms}.

\begin{lem}
\label{lem:RatioHomormophism}
Suppose $(\vee, \varphi)$ is a pivotal structure on a semisimple multitensor category $\cC$.
Defining for a simple $c\in \cC$ 
\begin{equation}
\label{eq:RatioHomomorphism}
\pi(\gr(c))
:=
\frac{\dim^\varphi_L(c)}{\dim^\varphi_R(c)}
\end{equation}
gives a well defined a homormophism $\pi: \cU \to \bbC^\times$.
If $(\vee, \varphi)$ is pseudounitary, then $\operatorname{im}(\pi)\subset \bbR_{>0}$.
\end{lem}
\begin{proof}
First, note that for a simple $c\in \cC$, $\dim^\varphi_L(c) \neq 0 \neq \dim_R^\varphi(c)$ by Corollary \ref{cor:NonzeroDimensionsForSimples}.
Next, if $c, d\in \cC$ are simple such that $c\otimes d \neq 0$, 
\begin{equation}
\label{eq:ComposableMorphismsMultiplicativeDimension}
\Dim_L^\varphi(c)\Dim_L^\varphi(d) = \Dim_L^\varphi(c\otimes d)
\qquad
\Longrightarrow
\qquad
\dim_L^\varphi(c\otimes d) = \dim_L^\varphi(c) \dim_L^\varphi(d),
\end{equation}
and similarly for the right dimension.
If moreover $\gr(c) = \gr(d)$, then $e:=\gr(c\otimes d^\vee)$ is an idempotent in $\cU$ (an identity morphism).
Since $(\vee, \varphi)$ restricted to $\cC_e$ gives a spherical tensor category, we have
$$
\dim_L^\varphi(c \otimes d^\vee) 
= 
\dim_R^\varphi(c \otimes d^\vee)
\qquad
\Longleftrightarrow
\qquad
\frac{\dim^\varphi_L(c)}{\dim^\varphi_R(c)}
=
\frac{\dim^\varphi_L(d)}{\dim^\varphi_R(d)},
$$
and $\pi$ is well-defined.
Now \eqref{eq:ComposableMorphismsMultiplicativeDimension} immediately implies that $\pi$ is a homomorphism.
The last claim is obvious.
\end{proof}

\subsection{Balanced duals}
\label{sec:BalancedDuals}

In this section, $\cC$ is a unitary multitensor category.
The following definition was suggested by Andr\'{e} Henriques.

\begin{defn}
\label{def:PhiBalancedSolutions}
Let $\pi : \cU \to \bbR_{>0}$ be a groupoid homomorphism.
Denote by $\Psi$ the linear functional in $\cC(1\to 1) \to \bbC$ which sends every minimal projection to $1_\bbC$.
A $\pi$-\emph{balanced dual} of $c\in \cC$ is a triple $(c^{\vee_\pi}, \ev^\pi_c, \coev^\pi_c)$ such that the morphisms $\ev^\pi_c, \coev^\pi_c$ satisfy the zig-zag axioms and the $\pi$-\emph{balancing condition}: 
for all $f\in \cC(c \to c)$ and $g\in \cU$
\begin{equation}
\label{eq:PhiBalancing}
\Psi\left(
\ev^\pi_c \circ (\id_{c^{\vee_\pi}}\otimes f_{g}) \circ (\ev^\pi_c)^\dag
\right)
=
\pi(g)
\cdot
\Psi\left(
(\coev^\pi_c)^\dag \circ (f_{g}\otimes \id_{c^{\vee_\pi}}) \circ \coev^\pi_c
\right).
\end{equation}
A unitary dual functor $\vee_\pi$ is called \emph{$\pi$-balanced} if \eqref{eq:PhiBalancing} holds for all $c\in \cC$, $f\in \cC(c\to c)$, and $g\in \cU$.

If $\pi(g)=1$ for all $g\in \cU$, we omit $\pi$ from the notation;
we simply say $(c^{\vee}, \ev_c, \coev_c)$ is a \emph{balanced} dual which satisfies the zig-zag axioms and the \emph{balancing condition}.
A dual functor $\vee$ is \emph{balanced} if \eqref{eq:PhiBalancing} holds with $\pi(g) =1$ for all $c\in \cC$, $f\in \cC(c\to c)$, and $g\in \cU$.
\end{defn}

Our next task is to construct a $\pi$-balanced unitary dual functor for every $\pi \in \Hom(\cU \to \bbR_{>0})$.
To do so, we will use the following facts from \cite{MR2091457}.

\begin{fact}[{\cite[Lem.~3.6]{MR2091457}}]
\label{fact:Faithful}
Suppose $(c^\vee, \ev_c, \coev_c)$ is an arbitrary dual of $c\in \cC$.
The positive map $\cC(c\to c) \to \cC(1\to 1)$ given by $f\mapsto \ev_c\circ (\id_{c^\vee} \otimes f) \circ \ev_c^\dag$ is \emph{faithful}:
for any $g\in \cC(c\to d)$, 
$$
\ev_c\circ (\id_{c^\vee} \otimes (g^\dag \circ g)) \circ \ev_c^\dag = 0
\Longleftrightarrow
(\id_{c^\vee} \otimes \circ g) \circ \ev_c^\dag = 0
\Longleftrightarrow
g=0.
$$
Similarly, $f\mapsto \coev_c^\dag\circ (f\otimes \id_{c^\vee}) \circ \coev_c$ is faithful.\footnote{
More is true:
given a morphism $\varepsilon \in \cC(c^\vee\otimes c\to 1_\cC)$, the map 
$f\mapsto \epsilon\circ (\id_{c^\vee} \otimes f) \circ \epsilon^\dag$ 
is faithful if and only if there is a morphism $\eta \in \cC(1_\cC\to c\otimes c^\vee)$ such that $(c^\vee, \varepsilon, \eta)$ is a dual of $c$.
This is proven in \cite[Lem.~3.6]{MR2091457} where the condition that $1_\cC$ is simple is never used.
A similar statement holds swapping $\varepsilon$ and $\eta$.
}
\end{fact}

\begin{fact}[{\cite[Lem.~3.7.ii]{MR2091457}, \cite[Prop.~2.6]{MR3308880}}]
\label{fact:DagAndVeeCompatibility}
For $a,b\in \cC$ with choices of arbitrary duals $(a^\vee, \ev_a, \coev_a), (b^\vee, \ev_b,\coev_b)$ respectively, the following are equivalent:
\begin{enumerate}[(1)]
\item
For all $f\in \cC(a\to b)$, $f^{\vee\dag} = f^{\dag\vee}$.
\item
For all $g\in \cC(b\to a)$, $g^{\vee\dag} = g^{\dag\vee}$.
\item
For all $a,b\in\cC$, $f\in \cC(a\to b)$, and $g\in \cC(b\to a)$,
\begin{equation}
\label{eq:EvaluationFunctional}
\ev_a \circ (\id_{a^\vee} \otimes (g\circ f)) \circ \ev_a^\dag
=
\ev_b \circ (\id_{b^\vee} \otimes (f\circ g)) \circ \ev_b^\dag.
\end{equation}
\item
For all $a,b\in\cC$, $f\in \cC(a\to b)$, and $g\in \cC(b\to a)$,
\begin{equation}
\label{eq:CoevaluationFunctional}
\coev_a^\dag \circ ((g\circ f) \otimes \id_{a^\vee} ) \circ \coev_a
=
\coev_b^\dag \circ ((f\circ g)\otimes \id_{b^\vee} ) \circ \coev_b.
\end{equation}
\end{enumerate}
\end{fact}
\begin{proof}
\mbox{}
\item[\underline{$(1)\Rightarrow (2)$:}]
Suppose (1) holds and $g\in \cC(b\to a)$.
Then $g^\dag \in \cC(a\to b)$, so
$$
g^{\dag\vee} 
=
(g^{\dag})^{\vee\dag\dag}
=
(g^{\dag})^{\dag\vee\dag}
=
g^{\vee\dag}.
$$

\item[\underline{$(2)\Rightarrow (1)$:}]
Analogous to $(1)\Rightarrow (2)$.

\item[\underline{$(2)\Leftrightarrow (3)$:}]
This is similar to the proof of \cite[Lem.~3.7.ii]{MR2091457}, which does not require that $1_\cC$ is simple or that $a=b$.
(Indeed, \cite[Prop.~2.6]{MR3308880} does not assume $a=b$.)
In more detail, observe that for $f\in \cC(a\to b)$ and $g\in \cC(b\to a)$,
$$
\ev_a \circ (\id_{a^\vee} \otimes (g\circ f)) \circ \ev_a^\dag
=
\ev_b \circ (g^\vee \otimes f) \circ \ev_a^\dag.
$$  
By faithfulness from Fact \ref{fact:Faithful}, the above is equal to the right hand side of \eqref{eq:EvaluationFunctional} if and only if 
$$
(\coev^\dag_b\otimes \id_a)
\circ
(\id_b \otimes g^\vee \otimes \id_a)
\circ
(\id_b \otimes \ev^\dag_a)
=
g
\qquad
\Longleftrightarrow
\qquad
g^{\vee\dag} = g^{\dag\vee}.
$$

\item[\underline{$(1)\Leftrightarrow (4)$:}]
This is similar to $(2)\Leftrightarrow (3)$ and left to the reader.
\end{proof}

\begin{prop}
\label{prop:ExistsUniqueBalancedDual}
Fix a groupoid homomorphism $\pi : \cU \to \bbR_{>0}$.
For each $c\in \cC$, there exists a unique $\pi$-balanced dual $(c^{\vee_\pi}, \ev^\pi_c, \coev^\pi_c)$ up to unique unitary isomorphism.
\end{prop}
\begin{proof}
We adapt the proof from \cite[Lem.~3.9]{MR2091457} and \cite[Prop.~2.2.15]{MR3204665}.

\item[\underline{Step 1:}]
Suppose $c\in \cC$ is simple, and let $(c^\vee, \ev_c, \coev_c)$ be a dual of $c$.
Then \eqref{eq:PhiBalancing} is satisfied if and only if $\pi(\gr(c))(\ev_c\circ \ev_c^\dag) = \coev_c^\dag \circ \coev_c$.
Since $\pi(\gr(c))>0$, we can scale $\ev_c$ and $\coev_c$ by inverse scalars to achieve this.
Moreover, this choice of scalar is unique up to a phase.
Hence the choice of $\ev_c^\pi$ and $\coev_c^\pi$ satisfying \eqref{eq:PhiBalancing} is unique up to a unique phase in $U(1)$.

\item[\underline{Step 2:}]
Suppose $c\in \cC$ is an orthogonal direct sum of $n$ objects isomorphic to the simple object $a\in \cC$.
Let $(a^{\vee_\pi}, \ev_a^\pi, \coev_a^\pi)$ be the unique $\pi$-balanced dual of $a$ from Step 1.
Suppose $c^{\vee_\pi}\in \cC$ can be equipped with an evaluation and coevaluation which make it a dual for $c$.
Pick $n$ isometries $v_1, \dots, v_n \in \cC(a \to c)$ and $w_1,\dots, w_n \in \cC(a^{\vee_\pi} \to c^{\vee_\pi})$ with orthogonal ranges so that $\sum_{i=1}^n v_i\circ v_i^\dag = \id_c$ and $\sum_{i=1}^n w_i \circ w_i^\dag = \id_{c^{\vee_\pi}}$.
Define 
$$
\ev_c^\pi := \sum_{i=1}^n \ev_a^\pi \circ (w_i^\dag \otimes v_i^\dag)
\qquad
\text{and}
\qquad
\coev_c^\pi := \sum_{i=1}^n (v_i \otimes w_i) \circ \coev_a^\pi.
$$
It is clear that $\ev_c^\pi$ and $\coev_c^\pi$ satisfy the zig-zag axioms.
Moreover, for $v_k \circ v_\ell^\dag$, we calculate that
\begin{align}
\ev_c^\pi \circ (\id_{c^{\vee_\pi}} \otimes (v_k \circ v_\ell^\dag)) \circ (\ev_c^\pi)^\dag
&=
\sum_{i,j=1}^n \ev_a^\pi \circ (w_i^\dag \otimes v_i^\dag) \circ (\id_{c^{\vee_\pi}} \otimes (v_k \circ v_\ell^\dag)) \circ (w_j \otimes v_j)\circ (\ev_a^\pi)^\dag
\notag
\\&=
\sum_{i,j=1}^n \ev_a^\pi  \circ (\id_{c^{\vee_\pi}} \otimes (v_i^\dag\circ v_k \circ v_\ell^\dag\circ v_j)) \circ  (\ev_a^\pi)^\dag
\label{eq:PhiBalancedEvaluationFunctional}
\\&=
\delta_{k=\ell} (\ev_a^\pi \circ (\ev_a^\pi)^\dag)
\qquad\qquad\text{and similarly,}
\notag
\\
(\coev_c^\pi)^\dag \circ ((v_k \circ v_\ell^\dag)\otimes \id_{c^{\vee_\pi}}) \circ \coev_c^\pi
&=
\delta_{k=\ell} ((\coev_a^\pi)^\dag \circ \coev_a^\pi).
\label{eq:PhiBalancedCoevaluationFunctional}
\end{align}
Hence \eqref{eq:PhiBalancing} is satisfied since it was true for $a\in\cC$ by Step 1.

Now suppose in addition to $(c^{\vee_\pi}, \ev_c^\pi, \coev_c^\pi)$, we have a second dual $(c^\vee, \ev_c, \coev_c)$ which is $\pi$-balanced.
Pick $n$ isometries $x_1, \dots, x_n \in \cC(a \to c)$ and $y_1,\dots, y_n \in \cC(a^{\vee_\pi} \to c^\vee)$ with orthogonal ranges so that $\sum_{i=1}^n x_i\circ x_i^\dag = \id_c$ and $\sum_{i=1}^n y_i \circ y_i^\dag = \id_{c^\vee}$.
Then there is a unitary $u \in \cC(c\to c)$ such that $ux_i = v_i$ for $i=1,\dots, n$.
Define the scalars $u_{ij}:= x_i^\dag\circ u\circ x_j \in \cC(a\to a) \cong\bbC$ so that $(u_{ij})_{i,j=1}^n\in M_n(\bbC)$ is unitary.
Define $U = \sum_{i,j} u_{ij} (w_j \circ y_i^\dag) \in \cC(c^\vee\to c^{\vee_\pi})$ which is necessarily unitary by construction.
Then we have
$$
v_j = u\circ x_j =  \sum_i u_{ij} x_i
\qquad
\text{and}
\qquad
U\circ y_i = \sum_j u_{ij} w_j,
$$
which implies
$$
\sum_i x_i \otimes (U\circ y_i) 
= 
\sum_{i,j} 
u_{i,j} 
(x_i \circ w_j)
=
\sum_j
v_j \otimes w_j
\Longrightarrow
(\id_c \otimes U) \circ \coev_c =\coev_c^\pi.
$$
By Remark \ref{rem:CoevDeterminesEv}, we have that the unique isomorphism $c^\vee \to c^{\vee_\pi}$ is equal to $U$ and is necessarily unitary.

\item[\underline{Step 3:}]
Suppose $c\in \cC$ is arbitrary.
Decompose $c$ into an orthogonal direct sum of isotypic components and apply Step 2.
\end{proof}

\begin{lem}
\label{lem:TensorProductOfBalancedDuals}
Fix a groupoid homomorphism $\pi : \cU \to \bbR_{>0}$.
Let $(c^{\vee_\pi}, \ev_c^\pi, \coev_c^\pi)$ and $(d^{\vee_\pi}, \ev_d^\pi, \coev_d^\pi)$ be the $\pi$-balanced duals of $c$ and $d$ respectively.
For $c\otimes d\in \cC$, the dual 
$$
(
d^{\vee_\pi}\otimes c^{\vee_\pi}
, 
\ev^\pi_d \circ (\id_{d^{\vee_\pi}}\otimes \ev^\pi_c\otimes \id_d)
, 
(\id_c \otimes \coev^\pi_d\otimes \id_{c^{\vee_\pi}})\circ\coev^\pi_c
)
$$ 
is $\pi$-balanced.
\end{lem}
\begin{proof}
We prove the lemma in the case $c,d$ are homogeneous following \cite[Lem.~3.11.i]{MR2091457}, and we leave the rest of the details to the reader.
In this case, for $f \in \cC(c\otimes d \to c\otimes d)$, $f$ is homogeneous, and we define
\begin{align*}
E_c^L(f) 
&:=
(\ev_c^\pi \otimes \id_d)\circ (\id_{c^{\vee_\pi}} \otimes f) \circ ((\ev_c^\pi)^\dag \otimes \id_d)
\in \cC(d\to d)
\\
E_d^R(f) 
&:=
( \id_c \otimes (\coev_d^\pi)^\dag)\circ (f\otimes \id_{d^{\vee_\pi}}) \circ ( \id_c \otimes \coev_d^\pi)
\in \cC(d\to d),
\end{align*}
which are also homogeneous morphisms.
We then calculate
\begin{align*}
&\Psi\left(
\left(
\ev^\pi_d \circ (\id_{d^{\vee_\pi}}\otimes \ev^\pi_c\otimes \id_d)
\right)
\circ
(\id_{d^{\vee_\pi} \otimes c^{\vee_\pi}} \otimes  f)
\circ
\left(
\ev^\pi_d \circ (\id_{d^{\vee_\pi}}\otimes \ev^\pi_c\otimes \id_d)
\right)^\dag
\right)
\\&=
\Psi\left(
\ev_d^\pi \circ (\id_{d^{\vee_\pi}} \otimes E^L_c(f)) \circ (\ev_d^\pi)^\dag
\right)
\\&=
\pi(\gr(d))
\cdot
\Psi\left(
(\coev_d^\pi)^\dag \circ  (E^L_c(f)\otimes \id_{d^{\vee_\pi}}) \circ \coev_d^\pi
\right)
\\&=
\pi(\gr(d))
\cdot
\Psi\left(
\ev_c^\pi \circ (\id_{c^{\vee_\pi}} \otimes E^R_d(f)) \circ (\ev_c^\pi)^\dag
\right)
\\&=
\pi(\gr(c))\pi(\gr(d))
\cdot
\Psi\left(
(\coev_c^\pi)^\dag \circ (\id_{c^{\vee_\pi}} \otimes E^R_d(f)) \circ \coev_c^\pi
\right)
\\&=
\pi(\gr(c\otimes d))
\cdot
\Psi\left(
\left(
(\id_c \otimes \coev^\pi_d\otimes \id_{c^{\vee_\pi}})\circ\coev^\pi_c
\right)^\dag
\circ
(f\otimes \id_{d^{\vee_\pi}\otimes c^{\vee_\pi}})
\left(
(\id_c \otimes \coev^\pi_d\otimes \id_{c^{\vee_\pi}})\circ\coev^\pi_c
\right)
\right).
\end{align*}
Thus the dual 
$(d^{\vee_\pi}\otimes c^{\vee_\pi}, \ev^\pi_d \circ (\id_{d^{\vee_\pi}}\otimes \ev^\pi_c\otimes \id_d), (\id_c \otimes \coev^\pi_d\otimes \id_{c^{\vee_\pi}})\circ\coev^\pi_c)$ 
is $\pi$-balanced.
\end{proof}

\begin{cor}
\label{cor:ExistsBalancedUnitaryDualFunctors}
For every $\pi \in \Hom(\cU \to \bbR_{>0})$, there exists a unique $\pi$-balanced unitary dual functor $\vee_\pi$ up to unique unitary monoidal natural isomorphism.
\end{cor}
\begin{proof}
By Proposition \ref{prop:ExistsUniqueBalancedDual}, for every $c\in \cC$, there exists a unique $\pi$-balanced dual $(c^{\vee_\pi}, \ev^\pi_c, \coev^\pi_c)$ up to unique unitary isomorphism.
By Lemma \ref{lem:TensorProductOfBalancedDuals} and Proposition \ref{prop:ExistsUniqueBalancedDual}, the canonical tensorator $\nu^\pi$ from \eqref{eq:CanonicalTensorator} is necessarily unitary.
It remains to prove $\vee_\pi$ is a dagger functor.
As in the proof of \cite[Lem.~3.9]{MR2091457}, 
By \eqref{eq:PhiBalancedEvaluationFunctional} and \eqref{eq:PhiBalancedCoevaluationFunctional}, equations \eqref{eq:EvaluationFunctional} and \eqref{eq:CoevaluationFunctional} hold, i.e., the positive linear maps 
\begin{align*}
f&\mapsto \ev^\pi_c \circ (\id_{c^{\vee_\pi}}\otimes f) \circ (\ev^\pi_c)^\dag \in \cC(1\to 1)
\\
f&\mapsto (\coev^\pi_c)^\dag \circ (f\otimes \id_{c^{\vee_\pi}}) \circ \coev^\pi_c \in \cC(1\to 1)
\end{align*}
are tracial (and faithful by Fact \ref{fact:Faithful}).
Hence by Fact \ref{fact:DagAndVeeCompatibility}, $\vee_\pi$ is a dagger functor.
\end{proof}

\begin{cor}
\label{cor:ClassificationOfUnitaryPivotalStructures}
The following are in canonical bijection:
\begin{enumerate}[(1)]
\item
Unitary dual functors up to unique unitary monoidal natural isomorphism
\item
$\Hom(\cU \to \bbR_{>0})$.
\end{enumerate}
\end{cor}
\begin{proof}
Suppose $\vee$ is a unitary dual functor, and let $\varphi$ be the canonical associated unitary pivotal structure, which is pseudounitary by Remark \ref{rem:UnitaryImpliesPseudounitary}.
Thus \eqref{eq:RatioHomomorphism} from Lemma \ref{lem:RatioHomormophism} gives us a function $\vee\mapsto \pi_{\vee}$ from unitary equivalence classes of unitary dual functors to $\Hom(\cU \to \bbR_{>0})$, which is injective by Lemmas \ref{lem:EquivalentPivotalStructures} and \ref{lem:UnitarilyEquivalentUnitaryDuals}.
Surjectivity now follows immediately from Corollary \ref{cor:ExistsBalancedUnitaryDualFunctors}, since it is easy to calculate that $\pi_{\vee_\pi} = \pi$ by \eqref{eq:PhiBalancing}.
\end{proof}

\begin{rem}
\label{rem:OtherGradingsInduceUnitaryDualFunctors}
Suppose $\cC$ is faithfully graded by the groupoid $\cG$.
Then for any $\pi \in \Hom(\cG \to \bbR_{>0})$, we get a unique lift $\widetilde{\pi} \in \Hom(\cU \to \bbR_{>0})$ using the canonical canonical groupoid surjection $\cU \twoheadrightarrow \cG$ from Remark \ref{rem:CanonicalGroupoidSurjection}.
Then the unique $\widetilde{\pi}$-balanced unitary dual functor is \emph{$\pi$-balanced}:
for all simple $c\in \cC$ with $\gr(c) = g\in \cG$ and all $f\in \cC(c\to c)$, 
$$
\Psi(
\ev^{\widetilde{\pi}}_c \circ (\id_{c^\vee}\otimes f) \circ (\ev^{\widetilde{\pi}}_c)^\dag
)
=
\pi(g)
\cdot
\Psi(
(\coev^{\widetilde{\pi}}_c)^\dag \circ (f\otimes \id_{c^\vee}) \circ \coev^{\widetilde{\pi}}_c
).
$$
\end{rem}

Choosing $\cG$ to be trivial or $\pi \in \Hom(\cU \to \bbR_{>0})$ trivial yields the following corollary.

\begin{cor}[{\cite[Thm.~4.7]{MR2091457} and \cite[\S4]{MR3342166}}]
\label{cor:CanonicalSphericalStructure}
A unitary multitensor category has a unique unitary spherical structure corresponding to $\pi = 1$ such that for all $f\in \cC(c\to c)$ and $g\in\cU$,
\begin{equation}
\label{eq:SphericalTrace}
\Psi\left(
\coev_c^\dag \circ (f_g\otimes \id_{c^\vee}) \circ \coev_c
\right)
= 
\Psi\left(
\ev_c \circ (\id_{c^\vee}\otimes f_g) \circ \ev_c^\dag
\right).
\end{equation}
\end{cor}

\begin{rem}
\label{rem:PerturbUnitaryDualFunctor}
Starting with a balanced unitary dual functor $\vee$, we can \emph{rescale} $\vee$ by a $\pi \in \Hom(\cU \to \bbR_{>0})$ to obtain a $\pi$-balanced unitary dual functor $\vee_\pi : \cC \to \cC^{\text{mop}}$ as follows.
For a homogeneous $c\in \cC$ with $\gr(c)=g\in \cU$, we define
\begin{equation}
\label{eq:RenormalizeCupsAndCaps}
\ev_c^\pi := \pi(g)^{1/4} \ev_c
\qquad
\qquad
\coev_c^\pi := \pi(g)^{-1/4} \coev_c.
\end{equation}
It is immediate that these renormalized maps satisfy the zig-zag axioms, and moreover, we see $(c^\vee, \ev_c^\pi , \coev_c^\pi)$ is $\pi$-balanced:
\begin{align*}
\ev^\pi_c
\circ
(\id_{c^\vee} \otimes f)
\circ 
(\ev^\pi_c)^\dag
&=
\pi(g)^{1/2}\cdot
\left(
\ev_c
\circ
(\id_{c^\vee} \otimes f)
\circ 
\ev_c^\dag
\right)
\\&=
\pi(g)^{1/2}
\cdot
\left(
\coev^\dag_c
\circ
(f\otimes \id_{c^\vee})
\circ 
\coev_c
\right)
\\&=
\pi(g)
\cdot
\left(
(\coev^\pi_c)^\dag
\circ
(f\otimes \id_{c^\vee})
\circ 
\coev_c^\pi
\right).
\end{align*}
However, notice that we have left $\nu$ and $\varphi$ \emph{unchanged}!
Indeed, if $c,d$ are homogeneous with $\gr(c)=g$ and $\gr(d)=h$, then $c\otimes d$ is homogeneous with $\gr(c\otimes d) = gh$, and
\begin{align*}
\nu^\pi_{c,d} 
&=
(\ev_d^\pi\otimes \id_{(c\otimes d)^\vee})
\circ 
(\id_{d^\vee}\otimes \ev^\pi_c \otimes \id_d \otimes \id_{(c\otimes d)^\vee})
\circ
(\id_{d^\vee\otimes c^\vee} \otimes \coev^\pi_{c\otimes d})
\\&=
\pi(g)^{1/4}\pi(h)^{1/4}\pi(gh)^{-1/4}
\cdot
\left(
(\ev_d\otimes \id_{(c\otimes d)^\vee})
\circ 
(\id_{d^\vee}\otimes \ev_c \otimes \id_d \otimes \id_{(c\otimes d)^\vee})
\circ
(\id_{d^\vee\otimes c^\vee} \otimes \coev_{c\otimes d})
\right)
\\&=
\nu_{c,d}.
\end{align*}
Similarly, $\gr(c^\vee) = g^{-1}$, and
$$
\varphi^\pi_c
=
((\coev^\pi_c)^\dag \otimes \id_{c^{\vee\vee}})
\circ
(\id_c \otimes \coev^\pi_{c^\vee})
=
\pi(g)^{-1/4} \pi(g^{-1})^{-1/4}\cdot
\left(
(\coev^\dag_c \otimes \id_{c^{\vee\vee}})
\circ
(\id_c \otimes \coev_{c^\vee})
\right)
=
\varphi_c.
$$
Since $\nu,\varphi$ are completely determined on homogeneous objects by naturality, we see that $\nu^\pi = \nu$ and $\varphi^\pi = \varphi$.

Conversely, starting with $\vee_\pi$ a $\pi$-balanced unitary dual functor, we can obtain a balanced unitary dual functor by rescaling evaluations and coevaluations on homogeneous objects $c\in \cC$ with $\gr(c) = g\in \cU$ by
\begin{equation}
\label{eq:SphericalizerCupsAndCaps}
\ev_c := \pi(g)^{-1/4} \ev^\pi_c
\qquad
\qquad
\coev_c := \pi(g)^{1/4} \coev^\pi_c.
\end{equation}
As before, $\nu$ and $\varphi$ remained unchanged by this scaling.
\end{rem}

\subsection{Bi-involutive structures}
\label{sec:Bi-involutive}

The notion of unitary dual functor is stronger than the similar notion of \emph{bi-involutive structure} from \cite[\S2.1]{MR3663592}.

\begin{defn}
An \emph{involutive structure} \cite{MR2861112} on a multitensor category $\cC$ consists of a conjugate-linear tensor functor $(\overline{\,\cdot\,},\nu) : \cC \to \cC^{\text{mp}}$\,\footnote{
Using the notation of \cite{1312.7188}, $\cC^{\text{mp}}$ denotes the tensor category obtained from $\cC$ by reversing the order of tensor product.
In other words, $(\overline{\,\cdot\,}, \nu): \cC \to \cC$ is conjugage-linear and anti-monoidal.
}
together with a monoidal natural isomorphism $\varphi: \id_\cC \to \overline{\overline{\,\cdot\,}}$\,\footnote{Monoidality is similar to \eqref{eq:EquivalentPivotalStructures}.}.
When $\cC$ is unitary, we call $(\overline{\,\cdot\,}, \nu,\varphi)$ a \emph{bi-involutive structure} \cite{MR3663592} if $\overline{\,\cdot\,}$ is a dagger functor and $\nu,\varphi$ are unitary.
\end{defn}

\begin{ex}
Complex conjugation gives a bi-involutive structure on the tensor $\Cstar$ category $\Hilb$ of separable Hilbert spaces.
\end{ex}

\begin{ex}
Similar to the previous example, complex conjugation gives a bi-involutive structure on $\Bim(M)$, the tensor $\Cstar$ category of separable $M-M$ bimodules where $M$ is any von Neumann algebra with the Connes fusion tensor product.
\end{ex}

By Proposition \ref{prop:VeeAndDagCompatibility} and Corollary \ref{cor:UnitaryPivotal}, every unitary dual functor $\vee_\pi: \cC \to \cC^{\text{mop}}$ on a unitary multitensor category $\cC$ gives a bi-involutive structure $(\overline{\,\cdot\,}^\pi, \nu^\pi, \varphi^\pi)$
as follows.
On objects, we define $\overline{c}^\pi := c^{\vee_\pi}$,
 and on morphisms $f\in \cC(a\to b)$, we define
$$
\overline{f}^{\pi}
:=
\begin{tikzpicture}[baseline=-.1cm]
	\draw (0,-1) -- (0,1);
	\draw (-.9,-1) -- (-.9,1.5);
	\draw (.9,1) -- (.9,-1.5);
	\roundNbox{fill=white}{(0,0)}{.3}{0}{0}{$f^\dag$}
	\roundNbox{fill=white}{(0,-1)}{.3}{.9}{0}{$(\ev_b^\pi)^\dag$}
	\roundNbox{fill=white}{(0,1)}{.3}{0}{.9}{$(\coev_a^\pi)^\dag$}
	\node at (1.1,-1) {\scriptsize{$\overline{b}^\pi$}};
	\node at (-1.1,1) {\scriptsize{$\overline{a}^\pi$}};
	\node at (-.2,-.5) {\scriptsize{$b$}};
	\node at (-.2,.5) {\scriptsize{$a$}};
\end{tikzpicture}
=
f^{\vee_\pi \dag}
=
f^{\dag \vee_\pi}.
$$
We take the canonical 
unitary monoidal natural isomorphisms $\nu^\pi = \{\nu^\pi_{a,b} : \overline{a}^\pi \otimes \overline{b}^\pi \to \overline{b\otimes a}^\pi\}$ and $\varphi^\pi: \id \Rightarrow \overline{\overline{\,\cdot\,}}^\pi$ induced by $\vee_\pi$.

\begin{rem}
Notice that the data of a bi-involutive structure $(\overline{\,\cdot\,}^\pi, \nu^\pi, \varphi^\pi)$ is weaker than that of the dual functor $\vee_\pi$ as we have \emph{forgotten} the  evaluation and coevaluation morphisms $\ev^\pi_c, \coev^\pi_c$ for $c\in \cC$.
Thus we \emph{cannot} recover $\vee_\pi$ from $(\overline{\,\cdot\,}^\pi, \nu^\pi, \varphi^\pi)$.
\end{rem}

In order to discuss unitary equivalence of bi-involutive structures, we first define the notion of a bi-involutive tensor functor from  \cite{MR3663592}.

\begin{defn}
\label{defn:Bi-involutiveTensorFunctor}
A \emph{bi-involutive tensor functor} between bi-involutive tensor $\Cstar$ categories 
$$
(F, \mu,\chi):
(\cC,\overline{\,\cdot\,}^\cC,\nu^\cC, \varphi^\cC) \to (\cD,\overline{\,\cdot\,}^\cD, \nu^\cD, \varphi^\cD)
$$ 
is a dagger tensor functor $(F, \mu)$ (our convention is $\mu_{a,b} : F(a)\otimes F(b) \to F(a\otimes b)$) equipped with a unitary natural isomorphism $\chi_c : F(\overline{c}) \to \overline{F(c)}$ which is monoidal with respect to $\mu,\nu^\cC, \nu^\cD$ and involutive with respect to $\varphi^\cC, \varphi^\cD$.
That is, the following diagrams commute:
$$
\begin{tikzcd}
F(\overline{a})\otimes F(\overline{b})
\ar[d, "\chi_a\otimes\chi_b"]
\ar[rr, "\mu_{\overline{a}, \overline{b}}"]
&&
F(\overline{a}\otimes \overline{b})
\ar[rr, "F(\nu^\cC_{a,b})"]
& &
F(\overline{b\otimes a})
\ar[d, "\chi_{b\otimes a}"]
\\
\overline{F(a)} \otimes \overline{F(b)}
\ar[rr, "\nu^\cD_{F(a),F(b)}"]
&&
\overline{F(b)\otimes F(a)}
\ar[rr, "\overline{\mu_{b,a}}"]
&&
\overline{F(b\otimes a)}
\end{tikzcd}
\qquad
\qquad
\begin{tikzcd}
F(a)
\ar[rr, "F(\varphi^\cC_a)"] \ar[d, "\varphi^\cD_{F(a)}"]  
&& 
F(\overline{\overline{a}}) 
\ar[d, "\chi_{\overline{a}}"]
\\
\overline{\overline{F(a)}}
&& 
\overline{F(\overline{a})}
\ar[ll, "\overline{\chi_{a}}"]
\end{tikzcd}
$$
\end{defn}

\begin{defn}
Two bi-involutive structures 
$(\overline{\,\cdot\,}^i, \nu^i, \varphi^i)$ on $\cC$ for $i=1,2$  
are \emph{unitarily equivalent} if there is an anti-monoidal involutive unitary natural isomorphism $\chi : \overline{\,\cdot\,}^1 \Rightarrow\overline{\,\cdot\,}^2$.
This means that $\chi$ satisfies the commutative diagrams in Definition \ref{defn:Bi-involutiveTensorFunctor} substituting $\cD$ with $\cC$ and $F: \cC\to \cD$ with $\id_\cC: \cC \to \cC$.
\end{defn}

\begin{rem}
\label{rem:AbsenceOfCanonicalUnitaryEquivalenceForBiInvolutive}
Given a bi-involutive structure $(\overline{\,\cdot\,}, \nu,\varphi)$ on a unitary multitensor category $\cC$, an autoequivalence $\chi \in \Aut((\overline{\,\cdot\,}, \nu,\varphi))$
consists of a unitary 
$\chi_c \in \cC(\overline{c} \to \overline{c})$  
for all $c\in \cC$
such that for all $a,b\in \cC$,
$\chi_{b\otimes a} \circ \nu_{a,b} = \nu_{a,b} \circ (\chi_a\otimes \chi_b)$
and
$\overline{\chi_a} = \chi_{\overline{a}}^{-1}$.
Similar to Remark \ref{rem:ClassifyingPivotalStructures} and Lemma \ref{lem:MonoidalNaturalAutomorphisms}, looking at simple objects, we get a canonical isomorphism between $\Aut((\overline{\,\cdot\,}, \nu,\varphi))$
and $\Hom(\cU^{\text{op}} \to U(1))$, but notice $g\mapsto g^{-1}$ gives an isomorphism $\cU \cong \cU^{\text{op}}$.
This means for any two bi-involutive structures $(\overline{\,\cdot\,}^i, \nu^i, \varphi^i)$ on $\cC$ for $i=1,2$, 
$\Hom((\overline{\,\cdot\,}^1,\nu^1, \varphi^1) \to (\overline{\,\cdot\,}^2,\nu^2,\varphi^2))$ is either empty or a torsor for $\Hom(\cU^{\text{op}} \to U(1))$.
Hence there is not a \emph{unique} unitary equivalence between two unitarily equivalent bi-involutive structures.
However, we will see in the proof of Corollary \ref{cor:UniqueBi-involutive} below that 
given two unitary dual functors, there is a \emph{canonical} unitary equivalence between their induced bi-involutive structures.
\end{rem}

\begin{ex}
\label{ex:SphericalAndTracialPivotalStructures}
When $(N,\tr)$ is a ${\rm II}_1$ factor with its canonical trace, there are two distinguished unitary dual functors that are often used in applications.
One is the balanced dual functor giving the canonical spherical structure corresponding to the trivial homomorphism $\pi=1$.
The other is obtained from the grading on $\bfBim(N)$ given by taking the ratio of the left to right von Neumann dimension.
When $N=R$, the hyperfinite ${\rm II}_1$ factor, this grading is faithful, since the fundamental group of $R$ is $\bbR_{>0}$ \cite{MR0009096}.
Taking the group homomorphism $\id: \bbR_{>0} \to \bbR_{>0}$ as in Remark \ref{rem:OtherGradingsInduceUnitaryDualFunctors} gives the \emph{tracial} unitary dual functor.
Calculating the universal grading group of $\bfBim(R)$ remains an important open question.
Interestingly, both the spherical and tracial unitary dual functors induce unitarily equivalent bi-involutive structures as was noted in \cite[Rem.~2.14]{1704.02035}.
\end{ex}

Motivated by Example \ref{ex:SphericalAndTracialPivotalStructures}, we now prove Corollary \ref{cor:UniqueBi-involutive}, which states that any two bi-involutive structures induced from unitary dual functors are canonically unitarily equivalent.
Notice that by Remark \ref{rem:AbsenceOfCanonicalUnitaryEquivalenceForBiInvolutive}, such a unitary equivalence is not unique.

\begin{proof}[Proof of Corollary \ref{cor:UniqueBi-involutive}]
Suppose $\vee_1$ and $\vee_2$ are two unitary dual functors on $\cC$, and let $\pi_1$ and $\pi_2$ be the corresponding homomorphisms in $\Hom(\cU \to \bbR_{>0})$.
While the unique monoidal natural isomorphism $\zeta: \vee_2 \Rightarrow \vee_1$ from \eqref{eq:CanonicalIntertwiner} is \emph{not} unitary,
it can be \emph{rescaled} as in Remark \ref{rem:PerturbUnitaryDualFunctor} to obtain a canonical unitary equivalence between the bi-involutive structures induced by $\vee_1$ and $\vee_2$.
Indeed, we define for a simple $c\in \cC$ with $\gr(c) = g$,
$$
\chi_c
:=
\left(
\frac{\pi_2(g)}{\pi_1(g)}
\right)^{1/4}
\begin{tikzpicture}[baseline=-.1cm]
	\draw[thick, red] (-.4,-.4) -- (-.4,0) arc (180:0:.2cm);
	\draw[thick] (.4,.4) -- (.4,0) arc (0:-180:.2cm);
	\node at (-.7,-.2) {\scriptsize{$c^{\vee_2}$}};
	\node at (.7,.2) {\scriptsize{$c^{\vee_1}$}};
\end{tikzpicture}
=
\left(
\frac{\pi_2(g)}{\pi_1(g)}
\right)^{1/4}
\zeta_c.
$$
The above formula is derived as follows.
First, we rescale $\vee_1$ and $\vee_2$ to get balanced dual functors $\vee_1^b$ and $\vee_2^b$ as in \eqref{eq:SphericalizerCupsAndCaps}.
By Corollary \ref{cor:ExistsBalancedUnitaryDualFunctors}, the unique monoidal natural isomorphism from \eqref{eq:CanonicalIntertwiner} $\zeta^b : \vee_2^b \Rightarrow \vee_1^b$ is necessarily unitary.
Notice now that $\chi_c = \zeta^b_c$ as morphisms from $c^{\vee_2} = c^{\vee_2^b}$ to $c^{\vee_1} = c^{\vee_1^b}$, as the rescaling procedure fixes the dual objects.
Hence $\chi$ is unitary.
It is now straightforward to verify that for all $a,b\in \cC$,
$\chi_{b\otimes a} \circ \nu_{a,b} = \nu_{a,b} \circ (\chi_a\otimes \chi_b)$
and
$\overline{\chi_a} = \chi_{\overline{a}}^{-1}$.
\end{proof}

\begin{rem}
Note that the canonical unitary equivalence $\chi$ from the proof of Corollary \ref{cor:UniqueBi-involutive} is the unique unitary natural isomorphism which can be obtained from the unique monoidal natural isomorphism $\zeta: \vee_1 \Rightarrow \vee_2$ by scaling the $\zeta_c$ for simple $c\in \cC$ by strictly positive real numbers.
Hence if we have three unitary dual functors $\vee_i$ for $i=1,2,3$ which induce bi-involutive structures $(\overline{\,\cdot\,}^i, \nu^i, \varphi^i)$ for $i=1,2,3$, then the canonical unitary equivalence $\chi^{12}$ composed with the canonical unitary equivalence $\chi^{23}$ is equal to the canonical unitary equivalence $\chi^{13}$.
\end{rem}

We finish this section by providing some important results on bi-involutive tensor functors.

\begin{prop}
\label{prop:PivotalFunctorBi-involutive}
Suppose $(F,\mu): (\cC, \vee) \to (\cD, \vee)$ is a dagger tensor functor between unitary multitensor categories equipped with unitary dual functors, and let $\varphi^\cC$ and $\varphi^\cD$ be the induced unitary pivotal structures.
The following are equivalent.
\begin{enumerate}[(1)]
\item
$(F,\mu)$ is pivotal with respect $\varphi^\cC$ and $\varphi^\cD$.
\item
The canonical isomorphism $\delta_c$ from \eqref{eq:CanonicalDualIso} is unitary for all $c\in \cC$.
\end{enumerate}
\end{prop}
\begin{proof}
For notational simplicity, we simply denote evaluations and coevaluations in this proof by $\ev$ and $\coev$.
Recall from \eqref{eqref:PivotalFunctor} that $(F,\mu)$ is pivotal if and only if for all $c\in \cC$, $\delta_c^\vee\circ \varphi^\cD_{F(c)} = \delta_{c^\vee}\circ F(\varphi^\cC_c)$.
This equality holds if and only if
$$
F(\coev_c^\dag) \circ \mu_{c, c^\vee} = \coev_{F(c)}^\dag \circ (\id_{F(c)} \otimes \delta_c)
$$
if and only if
$$
\delta_c 
= 
(\id_{F(c)^\vee}\otimes F(\coev_c^\dag))
\circ
(\id_{F(c)^\vee}\otimes \mu_{c, c^\vee})
\circ
(\ev_{F(c)}^\dag \otimes \id_{F(c^\vee)})
$$
if and only if $\delta_c=(\delta_c^{-1})^\dag$ is unitary.
\end{proof}

\begin{cor}
\label{cor:PivotalImpliesBi-involutive}
If either of the equivalent conditions in Proposition \ref{prop:PivotalFunctorBi-involutive} hold, then $(F,\mu,\delta)$ is bi-involutive.
\end{cor}
\begin{proof}
We must verify the diagrams in Definition \ref{defn:Bi-involutiveTensorFunctor} commute for $(F,\mu,\delta)$.
For the first diagram, one shows both composites from $F(a^\vee) \otimes F(b^\vee) \to F(b\otimes a)^\vee$ are equal to
\begin{align*}
([F(\ev_a) \circ \mu_{a^\vee, a}] \otimes \id_{F(b\otimes a)^\vee})
&\circ
(\id_{F(a^\vee)} \otimes [F(\ev_b) \circ \mu_{b,b^\vee}] \otimes \id_{F(a)} \otimes \id_{F(b\otimes a)^\vee})
\\&\circ
(\id_{F(a^\vee)\otimes F(b^\vee)} \otimes \mu_{b,a}^{-1} \otimes \id_{F(b\otimes a)^\vee})
\circ
(\id_{F(a^\vee)\otimes F(b^\vee)} \otimes \coev_{F(b\otimes a)}).
\end{align*}
We leave the details to the reader.
For the second diagram, just notice that this is exactly the pivotality condition \eqref{eqref:PivotalFunctor} when $\delta$ is unitary, as $\overline{\delta_c} = (\delta_c^\vee)^{-1}$.
\end{proof}

The following remark is based on a suggestion of Marcel Bischoff.

\begin{rem}
\label{rem:SimultaneouslyScalePivotalStructures}
Suppose that $\cC,\cD$ are unitary multitensor categories which are both faithfully graded by the groupoid $\cG$.
Suppose that we have unitary dual functors on $\cC$ and $\cD$, and let $(\overline{\,\cdot\,}^\cC,\nu^\cC, \varphi^\cC)$ and $(\overline{\,\cdot\,}^\cD, \nu^\cD, \varphi^\cD)$ be the induced bi-involutive structures.
Suppose that $(F, \mu): \cC \to \cD$ is a dagger tensor functor such that the canonical isomorphism $\delta_c$ from \eqref{eq:CanonicalDualIso} is unitary for all $c\in \cC$, so that $(F, \mu,\delta)$ is bi-involutive by Corollary \ref{cor:PivotalImpliesBi-involutive}.

Suppose now that $(F,\mu,\delta)$ preserves the grading groupoid $\cG$, i.e., $\gr(F(c)) = \gr(c)$ for all homogeneous $c\in \cC$.
Picking an arbitrary $\pi \in \Hom(\cG \to \bbR_{>0})$, we can renormalize the cups and caps by $\pi$ as in \eqref{eq:RenormalizeCupsAndCaps} from Remark \ref{rem:PerturbUnitaryDualFunctor} to get new unitary dual functors on $\cC$ and $\cD$ respectively.
Note that these new dual functors need not be $\pi$-balanced unless the unitary dual functors we started with were the canonical spherical ones.
However, for lack of better notation, we will denote the new evaluations and coevaluations by $\ev^\pi$ and $\coev^\pi$ respectively.

Notice that rescaling as in \eqref{eq:RenormalizeCupsAndCaps} leaves $\delta$ \emph{unchanged}!
Indeed, denoting the new canonical monoidal natural isomorphism by $\delta^\pi$ (again due to lack of better notation), for a homogeneous $c\in \cC$ with $\gr(c) = \gr(F(c))=g\in \cG$, we have
\begin{align*}
\delta^\pi_c
&=
([F(\ev^\pi_c) \circ \mu_{c^\vee,c}] \otimes \id_{F(c)^\vee})
\circ
(\id_{F(c)} \otimes \coev^\pi_{F(c)})
\\&=
\pi(g)^{1/4}
\pi(g)^{-1/4}
\cdot
\left(
([F(\ev_c) \circ \mu_{c^\vee,c}] \otimes \id_{F(c)^\vee})
\circ
(\id_{F(c)} \otimes \coev_{F(c)})
\right)
=
\delta_c.
\end{align*}
Since the bi-involutive structures of $\cC$ and $\cD$ did not change by Corollary \ref{cor:UniqueBi-involutive}, 
we conclude from Proposition \ref{prop:PivotalFunctorBi-involutive} that $(F,\mu)$ is pivotal with respect to the new unitary pivotal structures, as the pivotal functor condition \eqref{eqref:PivotalFunctor} still holds.
\end{rem}
 
\section{Planar algebras and projection categories}

Planar algebras come in many variants; among them are
\emph{unoriented} \cite[Def.~1.1.1]{JonesPANotes} (see also \cite{MR2559686}),
\emph{oriented} \cite[Def.~1.1.5]{JonesPANotes}  
\emph{disoriented} \cite{MR2496052}, and
\emph{shaded}  \cite[Def.~1.1.4]{JonesPANotes} (see also \cite{MR2679382,MR2972458}).

Shaded planar algebras were first defined in \cite{math.QA/9909027} to axiomatize the standard invariant of a finite index subfactor.
Since, they have been a valuable tool in the construction \cite{MR2979509,MR3314808} and classification \cite{MR3166042,1509.00038} of subfactor planar algebras as they give a generators and relations approach to subfactor theory.

The following folklore theorem is known to experts \cite{MR2559686,MR2811311,1207.1923,MR3405915,1607.06041}.
\begin{thm}
\label{thm:PAEquivalence}
There is an equivalence of categories\,\footnote{
\label{footnote:Truncate2Category}
Pairs $(\mathcal{C},X)$ form a $2$-category which is equivalent to a $1$-category in the following sense.
Between any two 1-morphisms, there is at most one 2-morphism, which is necessarily invertible when it exists.
We refer the reader to \cite[Lem.~3.5]{1607.06041} and the paragraph thereafter for more details on this subtelty.}
\[
\left\{
\text{\rm Oriented planar algebras}
\right\}
\longleftrightarrow
\left\{
\text{\rm Pairs $(\cC, X)$ with $\cC$ a pivotal category and $X\in\cC$ a generator}
\right\}
\]
Here, we call $X\in\cC$ a generator if every object of $\cC$ is isomorphic to a direct summand of a direct sum of tensor powers of $X$ and $X^\vee$.
\end{thm}

Theorem \ref{thm:PAEquivalence} holds for many sub-classes of planar algebras and pivotal categories.
We provide a helpful dictionary below:
\begin{center}
\begin{tabular}{c|c}
Planar algebras
&
Pivotal categories with generators
\\\hline
unoriented
&
symmetrically self-dual generator
\\
connected ($\dim(\cP_0)=1$)
&
$1_\cC$ simple
\\
2-shaded
&
partition $1_\cC=1_+ \oplus 1_-$ 
with generator $X = 1_+\otimes X \otimes 1_-$ 
\\
2-shaded connected
&
in addition to line above, $1_+, 1_-$ are simple
\\
semisimple 
& 
semisimple
\\
finite depth
&
finitely semisimple
\\
spherical
&
spherical
\\
$\Cstar$
&
$\Cstar$ with unitary pivotal structure
\end{tabular}
\end{center}
For example, we get an equivalence of categories between
finite depth subfactor (2-shaded connected spherical $\Cstar$) planar algebras 
and
pairs $(\cC,X)$ where
$\cC$ is a 
finitely semisimple unitary multifusion category with its canonical spherical structure (see Corollary \ref{cor:CanonicalSphericalStructure}) such that $1_\cC$ decomposes into simples as $1_\cC = 1_+\oplus 1_-$
and $X = 1_+\otimes X \otimes 1_-$ generates $\cC$.

\subsection{Correspondence between planar algebras and projection categories}

In all of the above cases, one can recover the original pivotal category from the planar algebra as the \emph{idempotent category}, or in the $\rm C^*$ cases as the \emph{projection category}. 
We only spell this out here in the case of 2-shaded planar $\rm C^*$ algebras.

\begin{defn}
The projection category of a shaded $\Cstar$ planar algebra $\cP_\bullet$
is the unitary multitensor category with unitary pivotal structure defined as follows:
\begin{itemize}
\item
The objects are the projections $p\in \cP_{n,\pm}$ ($p=p^\dag=p^2$) and the tensor product is horizontal juxtaposition (which is zero if the two shadings do not agree).
\item
The morphisms spaces are given for $p \in \cP_{n,\pm}$ and $q\in \cP_{n',\pm'}$ by  
$$
\Hom(p\to  q) = 
\delta_{\pm' = \pm}
\delta_{n' \equiv n \operatorname{mod} 2}
\left\{
x\in \cP_{(n+n')/2, \pm}
\,\,
\middle|
\,\,
x
=
\begin{tikzpicture}[baseline=-.1cm]
	\draw (0,-1.2) -- (0,1.2);
	\roundNbox{fill=white}{(0,0)}{.2}{0}{0}{$x$}
	\roundNbox{fill=white}{(0,.7)}{.2}{0}{0}{$q$}
	\roundNbox{fill=white}{(0,-.7)}{.2}{0}{0}{$p$}
	\node at (.2,.35) {\scriptsize{$n'$}};
	\node at (.2,1.05) {\scriptsize{$n'$}};
	\node at (.2,-.35) {\scriptsize{$n$}};
	\node at (.2,-1.05) {\scriptsize{$n$}};
	\node at (-.4,.7) {\scriptsize{$\star$}};
	\node at (-.4,0) {\scriptsize{$\star$}};
	\node at (-.4,-.7) {\scriptsize{$\star$}};
\end{tikzpicture}
\right\}
.
$$
\item
The adjoint is the dagger structure of the planar algebra; notice that if $x\in \Hom(p\to q)$, then $x^\dag\in \Hom(q\to p)$.
\item
We get a homomorphism $\pi : \cU \to \bbR_{>0}$ by taking the ratio of left to right traces as in \eqref{eq:RatioHomomorphism} from Lemma \ref{lem:RatioHomormophism}.
The $\pi$-balanced dual of $p\in \cP_{n,\pm}$ is given by the the $180^{\circ}$-rotation of $p$ in $\cP_{n,\mp}$, with evaluation and co-evaluation given by using cups and caps:
$$
\coev_p :=
\begin{tikzpicture}[baseline=-.1cm]
	\draw (-.3,.5) -- (-.3,-.2) arc (-180:0:.3cm) -- (.3,.5);
	\roundNbox{fill=white}{(.3,0)}{.2}{0}{0}{\rotatebox{180}{$p$}}
	\roundNbox{fill=white}{(-.3,0)}{.2}{0}{0}{$p$}
	\node at (-.7,0) {\scriptsize{$\star$}};
	\node at (.7,0) {\scriptsize{$\star$}};
	\node at (-.5,.35) {\scriptsize{$n$}};
	\node at (.5,.35) {\scriptsize{$n$}};
	\node at (-.5,-.35) {\scriptsize{$n$}};
	\node at (.5,-.35) {\scriptsize{$n$}};
\end{tikzpicture}
\qquad
\qquad
\ev_p :=
\begin{tikzpicture}[baseline=-.1cm]
	\draw (-.5,-.5) -- (-.5,.2) arc (180:0:.5cm) -- (.5,-.5);
	\roundNbox{fill=white}{(-.5,0)}{.2}{0}{0}{\rotatebox{180}{$p$}}
	\roundNbox{fill=white}{(.5,0)}{.2}{0}{0}{$p$}
	\node at (-.1,0) {\scriptsize{$\star$}};
	\node at (.1,0) {\scriptsize{$\star$}};
	\node at (-.7,.35) {\scriptsize{$n$}};
	\node at (.7,.35) {\scriptsize{$n$}};
	\node at (-.7,-.35) {\scriptsize{$n$}};
	\node at (.7,-.35) {\scriptsize{$n$}};
\end{tikzpicture}
.
$$
It is straightforward to calculate from the formula for the unitary pivotal structure in (2) of Proposition \ref{prop:VeeAndDagCompatibility} that $\varphi^\pi_p = \id_p$.
\end{itemize}
The generator corresponds to the unshaded-shaded strand 
$
\,\,\,
\begin{tikzpicture}[baseline=-.1cm]
	\fill[shaded] (0,-.2) rectangle (.2,.2); 
	\draw (0,-.2) -- (0,.2);
\end{tikzpicture} 
\in 
\cP_{1,+}
$.
\end{defn}

\begin{defn}
\label{defn:PAFromCategory}
Conversely, given a unitary multitensor category $\cC$ where $1_\cC = 1_+\oplus 1_-$ is an orthogonal decomposition (with $1_\pm$ not necessarily simple) together with a fixed $\pi \in \Hom( \cU \to \bbR_{>0})$ and generator $X = 1_+ \otimes X \otimes 1_-$, 
we obtain a shaded $\textrm{C}^*$ planar algebra $\cP_\bullet$ by a unitary version of \cite[\S4]{MR2811311}.
Let $(X^{\vee_\pi}, \ev_X^\pi, \coev_X^\pi)$ be the $\pi$-balanced dual of $X$ as in Proposition \ref{prop:ExistsUniqueBalancedDual}.
We use the notation
$$
X^{{\rm alt}\otimes n} := \underbrace{X \otimes X^{\vee_\pi}\otimes \cdots \otimes X^?}_{n\text{ tensorands}}
\qquad\qquad
(\overline{X}^\pi)^{{\rm alt}\otimes n} := \underbrace{X^{\vee_\pi} \otimes X\otimes \cdots \otimes (X^{\vee_\pi})^?}_{n\text{ tensorands}}
$$
where $X^? = X$ if $n$ is odd and $X^{\vee_\pi}$ if $X$ is even, and $(X^{\vee_\pi})^? = X$ if $n$ is even and $X^{\vee_\pi}$ if $n$ is odd.
The box spaces are defined for $n\geq 0$ by
\begin{equation}
\label{eq:BoxSpacesFromCategory}
\cP_{n,+}:=
\cC(X^{{\rm alt}\otimes n} \to  X^{{\rm alt}\otimes n})
\qquad
\qquad
\cP_{n,-}:=
\cC((X^{\vee_\pi})^{{\rm alt}\otimes n} \to (X^{\vee_\pi})^{{\rm alt}\otimes n}).
\end{equation}
The action of tangles is via the graphical calculus for pivotal categories; details appear in \cite[\S4]{MR2811311}.
For our purposes, we specify the actions of the following shaded planar tangles, which determines the action of every shaded planar tangle.
\begin{itemize}
\item
The strand is the identity morphsim
$
\begin{tikzpicture}[baseline = -.1cm]
	\fill[shaded] (0,-.3) -- (0,.3) -- (.2,.3) -- (.2,-.3);
	\draw (0,-.3) -- (0,.3);
\end{tikzpicture}
:=\id_X
$
and
$
\begin{tikzpicture}[baseline = -.1cm]
	\fill[shaded] (0,-.3) -- (0,.3) -- (-.2,.3) -- (-.2,-.3);
	\draw (0,-.3) -- (0,.3);
\end{tikzpicture}
:=\id_{X^{\vee_\pi}}
$
\item
Caps and cups which are shaded above are given by
$
\begin{tikzpicture}[baseline = 0cm]
	\fill[shaded] (-.1,0) -- (-.1,.3) -- (.5,.3) -- (.5,0) -- (.4,0) arc (0:180:.2cm);
	\draw[] (0,0) arc (180:0:.2cm);
\end{tikzpicture}
:=\ev_{X}^\pi
$
and
$
\begin{tikzpicture}[baseline = -.2cm, yscale = -1]
	\filldraw[shaded] (0,0) arc (180:0:.2cm);
\end{tikzpicture}
:=\coev_{X}^\pi
$
\item
Caps and cups which are shaded below are given by
$
\begin{tikzpicture}[baseline = 0cm]
	\filldraw[shaded] (0,0) arc (180:0:.2cm);
\end{tikzpicture}
:=(\coev_{X}^\pi)^\dag
$
and
$
\begin{tikzpicture}[baseline = -.2cm, yscale=-1]
	\fill[shaded] (-.1,0) -- (-.1,.3) -- (.5,.3) -- (.5,0) -- (.4,0) arc (0:180:.2cm);
	\draw[] (0,0) arc (180:0:.2cm);
\end{tikzpicture}
:=
(\ev_{X}^\pi)^\dag
$
\item
Vertical join is composition in $\cC$
$
\begin{tikzpicture}[baseline = -.15cm]
	\draw (0,-.55) -- (0,.65);
	\filldraw[thick, fill=white] (-.2,-.4) rectangle (.2,0);
	\filldraw[thick, fill=white] (-.2,.1) rectangle (.2,.5);
	\node at (0,-.2) {$f$};
	\node at (0,.3) {$g$};
\end{tikzpicture}
:=g\circ f
$
\item
Horizontal join is tensor product in $\cC$
$
\begin{tikzpicture}[baseline = -.1cm]
	\draw (0,-.35) -- (0,.35);
	\draw (.5,-.35) -- (.5,.35);
	\filldraw[thick, fill=white] (-.2,-.2) rectangle (.2,.2);
	\filldraw[thick, fill=white] (.3,-.2) rectangle (.7,.2);
	\node at (0,0) {$f$};
	\node at (.5,0) {$g$};
\end{tikzpicture}
:=f\otimes g.
$
\end{itemize}
The $\dag$-structure on $\cP_\bullet$ is just $\dag$ on morphisms in $\cC$.
\end{defn}

\subsection{Making closed loops scalar valued}
\label{sec:ScalarLoops}

For this section, we assume $\cC$ is a unitary multitensor category such that $\dim(\cC(1_\cC \to 1_\cC)) = r$ and that $\cC$ is faithfully graded by the groupoid $\cM_r$ consisting of the groupoid with $r$ objects and exactly one isomorphism between any two objects, which we can identify with the standard system of matrix units  $\{E_{ij}\}$ for $M_r(\bbC)$.

While $\cC$ has a canonical spherical structure from Corollary \ref{cor:CanonicalSphericalStructure}, it is not always the most relevant one for planar algebraic purposes, including graph planar algebra embedding.
Example \ref{ex:GPA} below discusses the graph planar algebra in detail.
Given an object $X\in \cC$ which partitions $1_\cC$ into a source and target summand, one may desire that closed loops count for scalars in the planar algebra associated to $(\cC, \vee, X)$ for some unitary dual functor $\vee$.
The conditions on $\vee$ are exactly provided by the recent article \cite{1805.09234}, which introduced the notion of standard solutions for the conjugate equations for a unitary multitensor category \emph{with respect to an object} $X \in \cC$ in the more general context of $\Cstar$ 2-categories.
We now rephrase their definition in our setup.

\begin{defn}[{\cite[Def.~7.25 and 7.29]{1805.09234}}]
\label{defn:StandardSolutionsWithRespectToX}
Suppose $X\in \cC$ such that there is an orthogonal decomposition $1_\cC = 1_+ \oplus 1_-$, which are not necessarily simple, such that $X = 1_+ \otimes X \otimes 1_-$.
Let $V_\pm$ be a set of representatives of the simple summands of $1_\pm$.

Let $D_X \in M_{V_+\times V_-}(\bbC)$ be the matrix whose $uv$-th entry is equal to the positive spherical dimension $\dim(u \otimes X \otimes v)$ using the canonical spherical structure from Corollary \ref{cor:CanonicalSphericalStructure}.
Let $d_X>0$ such that $d_X^2$ is the common Frobenius-Perron eigenvalue of $D_XD_X^T$ and $D_X^TD_X$.
Let $\lambda_+$ and $\lambda_-$ be the Frobenius-Perron eigenvectors for $D_XD_X^T$ and $D_X^T D_X$ which are normalized so that
$$
\sum_{u\in V_+} \lambda_+(u)^2
= 
1 
=
\sum_{v\in V_-} \lambda_-(v)^2.
$$
We define 
\begin{equation}
\label{eq:PerronFrobeniusEigenvector}
\lambda 
:=
\begin{pmatrix}
\lambda_+
\\
\lambda_-
\end{pmatrix}
\in \bbC^r.
\end{equation}
The unitary dual functor corresponding to the groupoid homomorphism $\pi_X: \cM_r \to \bbR_{>0}$ given by
\begin{equation}
\label{eq:StandardGroupoidHomFromFPdim}
\pi(E_{uv}) :=\left(\frac{\lambda(u)}{\lambda(v)}\right)^2.
\end{equation}
is called \emph{standard with respect to $X$}.
The unitary dual functor corresponding to the groupoid homomorphism $\pi_X^\ell: \cM_r \to \bbR_{>0}$ given by
\begin{equation}
\label{eq:LopsidedGroupoidHomFromFPdim}
\pi_X^\ell(E_{uv}) 
:=
\begin{cases}
d_X^{-1}
\left(
\frac{\lambda(u)}{\lambda(v)}
\right)^{2}
&\text{if $u\in V_+$ and $v\in V_-$}
\\
d_X
\left(
\frac{\lambda(u)}{\lambda(v)}
\right)^{2}
&\text{if $u\in V_-$ and $v\in V_+$}
\\
\left(
\frac{\lambda(u)}{\lambda(v)}
\right)^{2}
&\text{else.}
\end{cases}
\end{equation}
is called \emph{lopsided with respect to $X$}.
\end{defn}

The following lemma is immediate from \cite{1805.09234}.

\begin{lem}
\mbox{}
\begin{enumerate}[(1)]
\item
The 2-shaded $\Cstar$ planar algebra $\cP_\bullet$ corresponding to the triple $(\cC, X, \vee_{\pi_X})$
with 
$\id_X = 
\begin{tikzpicture}[baseline=-.1cm]
	\fill[shaded] (0,-.2) rectangle (.2,.2); 
	\draw (0,-.2) -- (0,.2);
\end{tikzpicture} 
\in 
\cP_{1,+}$
satisfies
\begin{equation}
\label{eq:StandardPA}
\begin{tikzpicture}[baseline=-.1cm]
	\draw[shaded] (0,0) circle (.3cm);
\end{tikzpicture}
=
d_X
\id_{1_+}
\qquad\qquad
\begin{tikzpicture}[baseline=-.1cm]
	\fill[shaded, rounded corners=5] (-.6,-.6) rectangle (.6,.6);
	\draw[fill=white] (0,0) circle (.3cm);
\end{tikzpicture}
=
d_X 
\id_{1_-}.
\end{equation}
\item
The 2-shaded $\Cstar$ planar algebra $\cP_\bullet^\ell$ corresponding to the triple $(\cC, X, \vee_{\pi_X^\ell})$
with 
$\id_X = 
\begin{tikzpicture}[baseline=-.1cm]
	\fill[shaded] (0,-.2) rectangle (.2,.2); 
	\draw (0,-.2) -- (0,.2);
\end{tikzpicture} 
\in 
\cP^\ell_{1,+}$
satisfies
\begin{equation}
\label{eq:LopsidedPA}
\begin{tikzpicture}[baseline=-.1cm]
	\draw[shaded] (0,0) circle (.3cm);
\end{tikzpicture}
=
\id_{1_+}
\qquad\qquad
\begin{tikzpicture}[baseline=-.1cm]
	\fill[shaded, rounded corners=5] (-.6,-.6) rectangle (.6,.6);
	\draw[fill=white] (0,0) circle (.3cm);
\end{tikzpicture}
=
d_X^2 
\id_{1_-}.
\end{equation}
\end{enumerate}
\end{lem}

\begin{rem}
The lopsided planar algebra is obtained from the standard planar algebra
as in Remark \ref{rem:SimultaneouslyScalePivotalStructures}
by replacing the standard $\ev_X$ and $\coev_X$ by $d_X^{1/2} \ev_X$ and $d_X^{-1/2} \coev_X$.
Notice this lopsided unitary pivotal structure varies slightly from the \emph{non-unitary} lopsided convention from \cite[\S1.1]{MR3254427}, which replaces $\ev_X$ and $\coev_X$ by $d_X \ev_X$ and $d_X^{-1} \coev_X$, but does \emph{not} scale $\ev^\dag_X$ nor $\coev^\dag_X$.
Often, it is computationally simpler to calculate a graph planar algebra embedding with respect to the non-unitary lopsided pivotal structure, as the number fields are more manageable.
It is still the case that non-unitary lopsided embeddings give standard unitary embeddings by \cite{MR3254427}.
\end{rem}

\begin{example}
\label{ex:GPA}
Let $\Gamma = (V_+, V_-, E)$ be a finite connected bipartite graph with even/$+$ vertices $V_+$, odd/$-$ vertices $V_-$, and edges $E$.
We consider an edge $\varepsilon \in E$ as directed from $+$ to $-$ with source $s(\varepsilon) \in V_+$ and target $t(\varepsilon) \in V_-$.
We write $\varepsilon^*$ for the same edge with the opposite direction.
Let $\lambda$ denote any Frobenius-Perron eigenvector of the adjacency matrix of $\Gamma$.

Denote by $n_\pm$ the number of vertices in $V_\pm$.
Let $\cM_\pm = \Hilb^{n_\pm}$, where $\Hilb$ is the category of finite dimensional Hilbert spaces considered as a semisimple $\rm C^*$ category.
We pick distinguished simples of $\cM_\pm$ which we name by the vertices in $V_\pm$.
Define $\cM = \cM_+ \oplus \cM_-$, which consists of one copy of $\Hilb$ for every vertex of $\Gamma$.

Now consider $\End^\dag(\cM)$, which we identify with the unitary multifusion category $\Mat^\dag_{n\times n}$.
The simple objects are the $E_{u,v}$ for $u,v \in V_+\amalg V_-$ with fusion rule $E_{u,v}\otimes E_{w,x} = \delta_{v=w} E_{u,x}$, and $1= \bigoplus_{v\in V_+\amalg V_-} E_{v,v}$.
It is straightforward to verify that the unique spherical structure from Corollary \ref{cor:CanonicalSphericalStructure} is given by $\ev_{E_{u,v}} : E_{v,u}\otimes E_{u,v} = E_{v,v} \hookrightarrow 1$ and $\coev_{E_{u,v}}: 1\twoheadrightarrow E_{u,u} = E_{u,v} \otimes E_{v,u}$.
Notice that the canonical spherical structure $\varphi$ satisfies $\varphi_{E_{u,v}} = \id_{E_{u,v}}$ for all vertices $u,v$.

Observe the universal grading groupoid $\cU$ of $\End^\dag(\cM)$ is $\cM_{n_++n_-}$.
By Corollary \ref{cor:ExistsBalancedUnitaryDualFunctors}, we get a canonical $\pi$-balanced unitary dual functor from the homomorphism $\pi : \cM_{n_++n_-} \to \bbR_{>0}$ given by \eqref{eq:StandardGroupoidHomFromFPdim}.
By direct computation as in the proof of Proposition \ref{prop:ExistsUniqueBalancedDual}, $\ev_{E_{u,v}}^\pi$ and $\coev_{E_{u,v}}^\pi$ are given by renormalizing the canonical $1$-balanced evaluation and coevaluation maps:
$$
\ev^\pi_{E_{u,v}}
:=
\left(
\frac{\lambda(u)}{\lambda(v)}
\right)^{1/2}
\ev_{E_{u,v}}
\qquad
\qquad
\coev^\pi_{E_{u,v}}
:=
\left(
\frac{\lambda(v)}{\lambda(u)}
\right)^{1/2}
\coev_{E_{u,v}}.
$$
Thus $\dim^\pi_L(E_{u,v}) = \frac{\lambda(u)}{\lambda(v)}$ and $\dim^\pi_R(E_{u,v}) = \frac{\lambda(v)}{\lambda(u)}$.

We now pick a distinguished dagger functor $F \in \End^\dag(\cM)$ together with its $\pi$-balanced dual
\begin{equation}
\label{eq:FunctorFromBipartiteGraph}
F = \bigoplus_{\varepsilon \in E} E_{s(\varepsilon),t(\varepsilon)}
\qquad
\qquad
F^{\vee_\pi} = \bigoplus_{\varepsilon \in E} E_{t(\varepsilon),s(\varepsilon)}.
\end{equation}
Since $\Gamma$ is connected, we see that $F$ generates $\End^\dag(\cM)$.
Define $1_+ := \bigoplus_{u\in V_+} E_{u,u}$ and $1_- := \bigoplus_{v\in V_-} E_{v,v}$, and note that $1= 1_+ \oplus 1_-$ is an orthogonal decomposition of the unit object such that $F = 1_+\otimes F \otimes 1_-$.

Let $\cG_\bullet$ be the corresponding shaded planar $\Cstar$ algebra corresponding to $(\End^\dag(\cM), F, \vee_\pi)$ under Theorem \ref{thm:PAEquivalence}, which was described in Definition \ref{defn:PAFromCategory}.
We may identify $\cG_{n,\pm}$ defined as in \eqref{eq:BoxSpacesFromCategory} as the complex vector space whose basis consists of the loops of length $2n$ on $\Gamma$ starting at a $\pm$ vertex in $V_\pm$.
For example, 
\begin{align*}
\cG_{n,+} 
&:= 
\Hom_{\End^\dag(\cM)}(F^{{\rm alt}\otimes n} \to  F^{{\rm alt}\otimes n})
\\&\cong
\Hom_{\End^\dag(\cM)}(1 \to (F\otimes F^{\vee_\pi})^{\otimes n}) 
\\&\cong 
\bigoplus_{v\in V_+}
\bigoplus_{\varepsilon_1,\dots, \varepsilon_{2n} \in E}
\Hom_{\End^\dag(\cM)}(E_{v,v} \to E_{s(\varepsilon_1),t(\varepsilon_1)} \otimes E_{t(\varepsilon_2), s(\varepsilon_2)} \otimes \cdots \otimes E_{s(\varepsilon_{2n-1}),t(\varepsilon_{2n-1})} \otimes E_{t(\varepsilon_{2n}), s(\varepsilon_{2n})} )
\\&\cong 
\bigoplus_{v\in V_+}
\operatorname{span}_\bbC\{\text{loops of length $2n$ based at $v$}\}.
\\&\cong 
\operatorname{span}_\bbC\{\text{loops of length $2n$ based at an even/+ vertex}\}.
\end{align*}
Under this identification, it is straightforward to verify that the actions of the shaded planar tangles described in Definition \ref{defn:PAFromCategory} exactly correspond to the actions of the shaded planar tangles for the planar algebra of the bipartite graph $\Gamma$ from \cite{MR1865703}.
\end{example}

From Theorem \ref{thm:PAEquivalence} and the discussion in Example \ref{ex:GPA}, we get the following.

\begin{prop}
\label{prop:GPAandEnd(M)}
Under the equivalence of categories in Theorem \ref{thm:PAEquivalence} for shaded $\Cstar$ planar algebras, the bipartite graph planar algebra $\cG_\bullet$ corresponds to the unitary multifusion category $\End^\dag(\cM)$ with (non-spherical!) unitary pivotal structure obtained from the standard groupoid homomorphism \eqref{eq:StandardGroupoidHomFromFPdim} with respect to $F_\Gamma$ as defined in \eqref{eq:FunctorFromBipartiteGraph}.
\end{prop}

\section{Spherical states on planar algebras and multitensor categories}
\label{sec:SphericalStates}

Motivated by the example of the graph planar algebra, we now define the notion of a spherical state with respect to a partition on a unitary multitensor category.
Such a spherical state can be chosen to `correct' for a non-spherical unitary pivotal structure on a unitary multitensor category if the corresponding $\pi \in \Hom(\cU \to \bbR_{>0})$ actually comes from a faithful grading by the groupoid $\cM_r$ 
consisting of the groupoid with $r$ objects and exactly one isomorphism between any two objects, which we can identify with
the standard system of matrix units  $\{E_{ij}\}$ for $M_r(\bbC)$.

\subsection{Evaluable planar algebras and spherical states}
\label{sec:EvaluablePlanarSubalgebras}

Below, we discuss sphericality and evaluability for \emph{semisimple} shaded planar algebras, i.e., each $\cP_{n,\pm}$ is a finite dimensional complex semisimple algebra under the usual multiplication in $\cP_\bullet$.

\begin{defn}
A shaded planar algebra is called \emph{evaluable} if $\dim(\cP_{0,\pm}) = 1$.
An evaluable shaded planar algebra is called \emph{spherical} if for all $x\in \cP_{1,+}$ the following two scalars in $\cP_{0,\pm}\cong \bbC$ (via mapping the empty diagrams to $1_\bbC$) agree:
$$
\begin{tikzpicture}[baseline=-.1cm]
	\fill[shaded, rounded corners=5] (-.8,-.8) rectangle (.5,.8);
	\filldraw[fill=white] (0,.3) arc (0:180:.3cm) -- (-.6,-.3) arc (-180:0:.3cm) -- (0,.3);
	\roundNbox{fill=white}{(0,0)}{.3}{0}{0}{$x$}
	\node at (-.4,0) {\scriptsize{$\star$}};
\end{tikzpicture}
=
\begin{tikzpicture}[baseline=-.1cm]
	\filldraw[shaded] (0,.3) arc (180:0:.3cm) -- (.6,-.3) arc (0:-180:.3cm) -- (0,.3);
	\roundNbox{fill=white}{(0,0)}{.3}{0}{0}{$x$}
	\node at (-.4,0) {\scriptsize{$\star$}};
\end{tikzpicture}
$$

For non-evaluable shaded planar algebras, \cite{JonesPANotes} defines sphericality in terms of a \emph{measure} on $\cP_{\bullet}$, which is a pair of linear functionals $\psi_\pm$ on the finite dimensional abelian complex semisimple algebras $\cP_{0,\pm}$.
A measure $(\psi_+,\psi_-)$ is called:
\begin{itemize}
\item
a \emph{state} if $\psi_\pm(p) \geq 0$ for every projection $p\in \cP_{0,\pm}$,
\item
a \emph{faithful state} if $\psi_\pm(p)> 0$ for every non-zero projection $p\in \cP_{0,\pm}$, and
\item
\emph{spherical} if for all $x\in \cP_{1,+}$,
$$
\psi_-\left(\,
\begin{tikzpicture}[baseline=-.1cm]
	\fill[shaded, rounded corners=5] (-.8,-.8) rectangle (.5,.8);
	\filldraw[fill=white] (0,.3) arc (0:180:.3cm) -- (-.6,-.3) arc (-180:0:.3cm) -- (0,.3);
	\roundNbox{fill=white}{(0,0)}{.3}{0}{0}{$x$}
	\node at (-.4,0) {\scriptsize{$\star$}};
\end{tikzpicture}
\,\right)
=
\psi_+\left(
\begin{tikzpicture}[baseline=-.1cm]
	\filldraw[shaded] (0,.3) arc (180:0:.3cm) -- (.6,-.3) arc (0:-180:.3cm) -- (0,.3);
	\roundNbox{fill=white}{(0,0)}{.3}{0}{0}{$x$}
	\node at (-.4,0) {\scriptsize{$\star$}};
\end{tikzpicture}
\,\,\,\,\right).
$$
\end{itemize}
\end{defn}

\begin{ex}
\label{ex:GPA-NotSpherical}
The graph planar algebra is in general not spherical.
For example, taking any edge $\varepsilon$ which connects two vertices of distinct weights, the projection $[\varepsilon \varepsilon^*] \in \cG_{1,+}$ has distinct left and right traces.
However, if we normalize the Frobenius-Perron eigenvector $\lambda$ so that $\sum_{u\in V_+}\lambda(u)^2 = 1 = \sum_{v\in V_-}\lambda(v)^2$, then $\psi(p_v) := \lambda(v)^2$ defines a spherical faithful state on $\cG_\bullet$ \cite[Prop.~3.4]{MR1865703}.
\end{ex}

\begin{remark}
An evaluable shaded planar algebra is spherical if and only if its pivotal projection multitensor category is spherical.
We will define the concepts of measure and (spherical faithful) state for a multitensor category in Section \ref{sec:Evaluable} below.
\end{remark}

\begin{remark}[{\cite[\S8]{MR1929335}}]
If $\cP_\bullet$ is a shaded planar ($\dag$-)algebra with a spherical faithful state,
then any evaluable planar ($\dag$-)subalgebra $\cQ_\bullet\subset \cP_\bullet$ is spherical.
If in addition $\cP_\bullet$ is $\Cstar$ with finite dimensional box spaces, then any evaluable $\cQ_\bullet\subset \cP_\bullet$ is a subfactor planar algebra.
\end{remark}

\subsection{Evaluable multitensor subcategories and spherical states}
\label{sec:Evaluable}

With the graph planar algebra in mind, we now define the notions of measure and state \emph{with respect to a partition $\Pi$} on a pivotal multitensor category which generalizes the similar notion for shaded planar algebras from \S\ref{sec:EvaluablePlanarSubalgebras}.

\begin{defn}
Given a multitensor category $\cC$, a \emph{partition $\Pi$ of $1_\cC$} consists of a family of mutually orthogonal projections $\{p\}\subset \End_\cC(1_\cC)$ such that $\sum_{p\in\Pi} p = 1$.

Given a multitensor category $\cC$ with a partition $\Pi$ of $1_\cC$, a \emph{measure with respect to $\Pi$} is a linear functional $\psi$ on the finite dimensional abelian semisimple $\bbC$-algebra $\End_\cC(1_\cC)$.
We call a measure $\psi$ with respect to $\Pi$:
\begin{itemize}
\item
a \emph{state} if $\psi(p) \geq 0$ for every idempotent $p\in \End_\cC(1_\cC)$ and $\psi(p) = 1$ for all $p\in \Pi$.
\item 
a \emph{faithful state} if $\psi$ is a state and $\psi(p) > 0$ for every non-zero idempotent $p\in \End_\cC(1_\cC)$.
(Note that if $\cC$ has simple tensor unit, there is a unique state, which is automatically faithful.)
\item
\emph{spherical} if $(\cC,\varphi)$ is pivotal and $\psi\circ \tr_L = \psi \circ \tr_R$ for all $c\in\cC$.
\end{itemize}
When $\Pi = \{\id_{1_\cC}\}$ is the trivial partition of $1_\cC$, we omit $\Pi$ from the notation and simply refer to measures and (spherical faithful) states.
\end{defn}

\begin{remark}
Note that in the case where $(\cC,\varphi)$ has non-simple tensor unit, having a faithful spherical state is additional structure over being pivotal.
However, if $(\cC, \varphi)$ is spherical, 
then setting $\Pi$ to be the set of all minimal projections in $\End_\cC(1_\cC)$, there is a canonical faithful spherical state with respect to $\Pi$ given by $\psi(p) = 1$ for every $p\in\Pi$.
\end{remark}

\begin{example}
By Example \ref{ex:GPA-NotSpherical}, the unitary multifusion category of projections of the graph planar algebra has a spherical faithful state with respect to the induced unitary dual functor and the partition $\Pi:=\{\sum_{u\in V_+} p_u,\sum_{v\in V_-} p_v\}$.
\end{example}

\begin{defn}
Suppose $\cC$ is a multitensor category with partition $\Pi$ of $1_\cC$.
For $p\in \Pi$, denote by $1_p$ the summand of $1_\cC$ corresponding to $p\in \End_\cC(1_\cC)$.
We call a unital multitensor subcategory $\cD \subset \cC$ \emph{evaluable with respect to $\Pi$} if for every $p\in \Pi$, $1_p \in \cD$ and $\End_\cD(1_p) = \bbC p$, i.e., 
$1_\cD = \bigoplus_{p\in \Pi} 1_p$ exhibits the decomposition of $1_\cD$ into simple objects of $\cD$.
\end{defn}

\begin{prop}
Let $(\cC,\varphi)$ be a pivotal multitensor category with $\Pi$ a partition of $1_\cC$.
Suppose $\psi$ is a spherical state on $\End_\cC(1_\cC)$ with respect to $\Pi$.
Then if $\cD$ is a unital multitensor subcategory of $\cC$ such that $\varphi_d \in \cD(d\to d)$ for all $d\in \cD$ and which is evaluable with respect to $\Pi$, then $(\cD,\varphi|_\cD)$ is spherical.
\end{prop}
\begin{proof}
Since $\varphi_d \in \cD(d\to d)$ for all $d\in \cD$, we have $\tr_L^\cC = \tr_L^\cD$ and $\tr_R^\cC=\tr_R^\cD$ for all $d\in \cD$.
Suppose $f\in \cD(c \to d)$.
Since $\cD$ is evaluable with respect to $\Pi$, $\Pi\subset \End_\cD(1_\cD)$ is the set of minimal projections (which may not be minimal in $\End_\cC(1_\cC)$).
Then for any $p,q\in \Pi$,
$$
\tr_L(p \otimes f \otimes q)
=
\psi(\tr_L(p \otimes f \otimes q))p
\qquad
\qquad
\tr_R(p \otimes f \otimes q)
=
\psi(\tr_R(p \otimes f \otimes q))q
$$
since $f$ is a morphism in $\cD$, and $\cD$ is evaluable with respect to $\Pi$.
Now since $\psi$ is a spherical state on $\End_\cC(1_\cC)$, we have
$\psi(\tr_L(p \otimes f \otimes q)) = \psi(\tr_R(p \otimes f \otimes q))$, and thus $\cD$ is spherical.
\end{proof}

\begin{example}
\label{ex:StandardWithRespectToX}
As we saw in \S\ref{sec:ScalarLoops}, one important way that partitions $\Pi$ of $1_\cC$ arise naturally is from picking a distinguished object $X$ in a unitary multitensor category whose source and range summands of $1_\cC$ are orthogonal.
That is $1_\cC = 1_+\oplus 1_-$ with $1_\pm$ not necessarily simple such that $X = 1_+\otimes X \otimes 1_-$.
We then set $\Pi = \{p_+ , p_-\}$ where $p_\pm\in \cC(1_\cC \to 1_\cC)$ is the orthogonal projection corresponding to the summand $1_\pm$.
\end{example}

\begin{ex}
\label{ex:SphericalFaithfulStateWithRespectToX}
Building on Definition \ref{defn:StandardSolutionsWithRespectToX} and Examples \ref{ex:GPA-NotSpherical} and \ref{ex:StandardWithRespectToX},
if $X=1_+\otimes X\otimes 1_-$ generates $\cC$ 
such that $1_\cC = 1_+ \oplus 1_-$ and $\cC$ is faithfully graded by $\cM_r$, then defining $\psi(v) := \lambda(v)^2$ for all simple summands $v\subset 1_\cC$ where $\lambda$ is defined as in \eqref{eq:PerronFrobeniusEigenvector} gives a spherical faithful state on $(\cC,\vee_{\text{standard}})$, where $\vee_{\text{standard}}$ is the standard unitary dual functor on $\cC$ with respect to $X$ from \eqref{eq:StandardGroupoidHomFromFPdim}.
One calculates that for all summands $u \subset 1_-$ and $v\subset 1_+$, 
$$
\psi(\tr_L^\pi(\id_{u\otimes X\otimes v}))
=
\lambda(u)\lambda(v) (D_X)_{u,v}
=
\psi(\tr_R^\pi(\id_{u\otimes X\otimes v})).
$$
As this equation is identical to \cite[(7.9)]{1805.09234}, we have that their canonical left/right states of $X$ on $\cC(X\to X)$ 
are given by $\omega_\ell = \psi \circ \tr^\pi_L$ and $\omega_r = \psi\circ \tr^\pi_R$.
\end{ex}

\subsection{The spherical state correction for non-balanced unitary dual functors}
\label{sec:SphericalCorrection}

For this section, $\cC$ is a unitary multitensor category which is faithfully graded by the groupoid $\cM_r$ consisting of 
the groupoid with $r$ objects and exactly one isomorphism between any two objects, which we can identify with the standard system of matrix units  $\{E_{ij}\}$ for $M_r(\bbC)$.
For notational simplicity, we write $\cC_{ij}$ for $\cC_{E_{ij}}$.
Moreover, we assume $1_\cC = \bigoplus_{i=1}^r 1_i$ is an orthogonal decomposition into simples, and we let $p_i \in \cC(1_\cC \to 1_\cC)$ be the minimal projection corresponding to the summand $1_i$.
We assume $\cC$ has the trivial partition $\Pi = \{\id_{1_\cC}\}$.

We now show that one can `correct' for a unitary dual functor on $\cC$ which is not balanced, but comes from a $\pi \in \Hom(\cM_r \to \bbR_{>0})$.
This can be viewed as a generalization of \cite[Thm.~3.1]{MR1865703} for the bipartite graph planar algebra (see also Example \ref{ex:GPA-NotSpherical}) and \cite[Thm.~7.39]{1805.09234}.

\begin{ex}
Suppose $\cC$ is a (unitary) $r\times r$ multifusion category, and let $\pi : \cU \to \bbC^\times$ be a homomorphism.
For every subgroup $\cH \subseteq \cU$ corresponding to the automorphisms of a single object, we must have $\pi(\cH)\subset U(1)$.
When $\cC$ is unitary and $\pi : \cU \to \bbR_{>0}$, we have $\pi(\cH)\subseteq U(1)\cap \bbR_{>0} = \{1\}$.
Thus the canonical groupoid surjection $\cU \twoheadrightarrow \cM_r$ from Remark \ref{rem:CanonicalGroupoidSurjection} induces an isomorphism $\Hom(\cU \to \bbR_{>0})\cong\Hom(\cM_r \to \bbR_{>0})$.
\end{ex}

\begin{fact}
\label{fact:BijectionForMatrixUnitHomomorphisms}
An element $\pi \in \Hom(\cM_r \to \bbR_{>0})$ is completely determined by its values on $E_{i+1, i}$ for $i=1,\dots, r-1$, which can be arbitrary.
Hence $\pi \mapsto (\pi(E_{i+1,i}))_{i=1}^{r-1}$ is a bijection
$\Hom(\cM_r \to \bbR_{>0}) \cong \bbR_{>0}^{r-1}$.
\end{fact}

\begin{lem}
\label{lem:RatioOfDimensionsBijection}
The function
$$
\psi \mapsto \left(\frac{\psi(p_{i+1})}{\psi(p_{i})}\right)_{i=1}^{r-1}
$$
is a bijection between faithful states $\psi$ on $\cC$ with respect to $\Pi$ and $\bbR_{>0}^{r-1}$.
\end{lem}
\begin{proof}
First, consider the set $\cF$ of all linear functionals $\psi:\cC(1_\cC \to 1_\cC) \to \bbC$ such that $\psi(p_i) \neq 0$ for all $i=1,\dots, r$.
The map $\cF \ni \psi \mapsto \left(\frac{\psi(p_{i+1})}{\psi(p_{i})}\right)_{i=1}^{r-1}\in \bbR_{>0}^{r-1}$ is clearly well defined and surjective.
Moreover, we see that $\psi_1$ and $\psi_2$ map to the same element of $\bbR_{>0}^{r-1}$ if and only if they are proportional:
$$
\frac{\psi_1(p_{i+1})}{\psi_1(p_{i})}
=
\frac{\psi_2(p_{i+1})}{\psi_2(p_{i})}
\qquad
\forall i=1,\dots,r-1
\qquad
\Longleftrightarrow
\qquad
\frac{\psi_2(p_i)}{\psi_1(p_{i})}
=
\frac{\psi_2(p_{i+1})}{\psi_1(p_{i+1})}
\qquad
\forall i=1,\dots,r-1.
$$
Finally, given a $\psi \in \cF$, there is exactly one faithful state with respect to $\Pi$ which is proportional to it.
\end{proof}

\begin{thm*}[Theorem \ref{thm:UniqueSphericalFaithfulState}]
Given a $\pi \in \Hom(\cM_r \to \bbR_{>0})$, there is a unique spherical faithful state $\psi^\pi$ with respect to $\Pi$ for $(\cC,\vee_\pi,\nu^\pi, \varphi^\pi)$.
\end{thm*}
\begin{proof}
\mbox{}
\item[\underline{Step 1:}]
Suppose $c\in \cC_{ij}$ is simple.
Then 
$$
\psi^\pi(\tr_L^\pi(\id_c))=\psi^\pi(\tr_R^\pi(\id_c))
\qquad
\Longleftrightarrow
\qquad
\psi^\pi(p_j) \dim_L^\pi(c) = \psi^\pi(p_i) \dim_R^\pi(c),
$$
which is equivalent to
$$
\frac{\psi^\pi(p_i)}{\psi^\pi(p_j)}
=
\pi(E_{ij})
\underset{\text{\eqref{eq:RatioHomomorphism}}}{=}
\frac{\dim_L^\pi(c)}{\dim_R^\pi(c)}.
$$
Notice that both $\pi$ and $E_{ij} \mapsto \frac{\psi^\pi(p_i)}{\psi^\pi(p_j)}$ are groupoid homomorphisms, so the above equality holds for all simple $c\in \cC_{ij}$ for all $i,j=1,\dots, r$ if and only if it holds for all simple $c\in \cC_{i+1, i}$ for all $i=1,\dots, r-1$.
This is equivalent to both homomorphisms corresponding to the same element of $\bbR_{>0}^{r-1}$ under the bijections from Fact \ref{fact:BijectionForMatrixUnitHomomorphisms} and Lemma \ref{lem:RatioOfDimensionsBijection} respectively.
Hence there is a unique choice of $\psi^\pi$ which works.

\item[\underline{Step 2:}]
Suppose $c\in \cC$ is an orthogonal direct sum of $n$ objects isomorphic to the simple object $a\in \cC$ and $f\in \cC(a\to a)$.
Pick $n$ isometries $v_1, \dots, v_n \in \cC(a \to c)$ with orthogonal ranges so that $\sum_{i=1}^n v_i\circ v_i^\dag = \id_c$.
Then for $\psi^\pi$ defined in Step 1,
\begin{align*}
\psi^\pi(
\tr_L^\pi(f)
)
&=
\sum_{i=1}^n
\psi^\pi(
\tr_L^\pi(v_i\circ v_i^\dag \circ f)
)
=
\sum_{i=1}^n
\psi^\pi(
\tr_L^\pi(v_i^\dag \circ f\circ v_i)
)
\\&=
\sum_{i=1}^n
\psi^\pi(
\tr_R^\pi(v_i^\dag \circ f\circ v_i)
)
=
\sum_{i=1}^n
\psi^\pi(
\tr_L^\pi(v_i\circ v_i^\dag \circ f)
)
=
\psi^\pi(
\tr_R^\pi(f)
).
\end{align*}

\item[\underline{Step 3:}]
Suppose $c\in \cC$ and $f\in \cC(c\to c)$ are arbitrary.
Decompose $c$ into an orthogonal direct sum of isotypic components and apply Step 2.
\end{proof}

\bibliographystyle{amsalpha}
{\footnotesize{
\bibliography{../../../../bibliography}
}}
\end{document}